\documentclass[oneside,a4paper,reqno]{amsart}
\usepackage{pdfsync, enumerate}
\usepackage{stmaryrd}
\usepackage{mathrsfs}
\usepackage{multicol}
\usepackage{amsmath, amsthm, amscd, amssymb, latexsym, eucal}
\usepackage[all]{xy}
\def\serieslogo@{} \def\@setcopyright{} \makeatother

\usepackage{multienum}

\usepackage[colorinlistoftodos]{todonotes}

\usepackage{boondox-calo}

\usepackage{hyperref}
\usepackage{color}
\usepackage{cite}

\usepackage{tabularx}

\usepackage{tikz-cd}
\usepackage{quiver}

\usepackage[capitalise, noabbrev, nameinlink]{cleveref}

\crefname{subsection}{Subsection}{Subsections}
\crefname{exam}{Example}{Examples}
\crefname{reform}{Reformulation}{Reformulations}

\makeatletter
\renewcommand*\env@matrix[1][c]{\hskip -\arraycolsep
	\let\@ifnextchar\new@ifnextchar
	\array{*\c@MaxMatrixCols #1}}
\makeatother

\usepackage{color}

\hypersetup{hidelinks}

\numberwithin{equation}{section}
\newtheorem{thm}{Theorem}[section]
\newtheorem*{main-thm}{Main Theorem}
\newtheorem*{Auslander-thm}{Auslander's Theorem}

\newtheorem{cor}[thm]{Corollary}
\newtheorem{lem}[thm]{Lemma}
\newtheorem{prop}[thm]{Proposition}

\theoremstyle{definition}
\newtheorem{defn}[thm]{Definition}
\newtheorem{rem}[thm]{Remark}
\newtheorem{exam}[thm]{Example}

\newtheorem*{problem}{Problem}

\newtheorem{reform}[thm]{Reformulation}

\newcommand{\lxr}{\longrightarrow}

\newcommand{\epic}{\twoheadrightarrow}

\newcommand{\A}{\mathscr A}
\newcommand{\B}{\mathscr B}
\newcommand{\C}{\mathscr C}

\newcommand{\mq}{\mathsf{q}}
\newcommand{\mi}{\mathsf{i}}

\newcommand{\ml}{\mathsf{l}}
\newcommand{\me}{\mathsf{e}}

\newcommand{\ol}[1]{\overline{#1}}

\newcommand{\la}{\langle}
\newcommand{\ra}{\rangle}

\newcommand{\monicc}{\hookrightarrow}

\newcommand{\oo}{\mathcal{o}}

\newcommand{\GG}{\mathcal{G}}

\newcommand{\BQ}{ \mathcal{B}_{ Q} }

\newcommand{\NN}{ \mathcal{N} }

\def\a{\alpha}
\def\b{\beta}
\def\g{\gamma}
\def\d{\delta}
\def\e{\varepsilon}
\def\z{\zeta}

\def\k{\kappa}
\def\l{\lambda}

\def\h{\eta}
\def\th{\theta}

\newcommand{\Ga}{\Gamma}

\DeclareMathOperator*{\Ker}{\mathsf{Ker}}

\DeclareMathOperator{\pd}{\mathsf{pd}}
\DeclareMathOperator*{\id}{\mathsf{id}}

\DeclareMathOperator{\fpd}{\mathsf{fin.dim}}

\DeclareMathOperator*{\findim}{\mathsf{fin.dim}}

\DeclareMathOperator*{\smod}{\mathsf{mod}-\!}
\DeclareMathOperator*{\sMod}{\mathsf{Mod}-\!}

\DeclareMathOperator*{\Proj}{\mathsf{Proj}}

\DeclareMathOperator*{\Tor}{\mathsf{Tor}}

\DeclareMathOperator{\idim}{\mathsf{id}}

\DeclareMathOperator{\tip}{\mathrm{Tip}}
\DeclareMathOperator{\ntip}{\mathrm{NonTip}}
\DeclareMathOperator{\ctip}{\mathrm{CTip}}

\newcommand{\iden}{\operatorname{Id}\nolimits}

\newcommand{\La}{\Lambda}

\newsavebox{\proofbox}
\savebox{\proofbox}{\begin{picture}(7,7)%
	\put(0,0){\framebox(7,7){}}\end{picture}}

\newcommand\isomto{\stackrel{\sim}{\smash{\longrightarrow}\rule{0pt}{0.4ex}}}

\newcommand{\ttt}{\mathfrak{t}}
\newcommand{\sss}{\mathfrak{s}}

\usepackage{latexsym}
\usepackage{pstricks}
\usepackage{comment}

\begin{document}


\title[Monomial arrow removal]
{Monomial arrow removal and the finitistic dimension conjecture}

\author[Erdmann]{Karin Erdmann}
\address{K.~Erdmann\\
	Mathematical Institute\\ 
	24--29 St.\ Giles\\
	Oxford OX1 3LB\\ 
	England}
\email{erdmann@maths.ox.ac.uk}

\author[Giatagantzidis]{Odysseas Giatagantzidis}
\address{O.~Giatagantzidis\\
	Dept of Mathematics\\
	AUTh\\
	54124 Thessaloniki\\
	Greece}
\email{odysgiat@math.auth.gr}

\author[Psaroudakis]{Chrysostomos Psaroudakis}
\address{C.~Psaroudakis\\
	Dept of Mathematics\\
	AUTh\\
	54124 Thessaloniki\\
	Greece}
\email{chpsaroud@math.auth.gr}

\author[Solberg]{\O yvind Solberg}
\address{\O.~Solberg\\
	Dept of Mathematical Sciences\\
	NTNU\\
	N-7491 Trondheim\\
	Norway}
\email{oyvind.solberg@math.ntnu.no}

\date{\today}

\keywords{%
	Finitistic dimension, Bound quiver algebras, Monomial arrow removal, Cleft extensions, (non-commutative) Gr\"{o}bner bases}

\subjclass[2010]{%
	18E, 
	16E30, 
	16E65; 
	16E10, 
	16G
}

\dedicatory{Dedicated to Claus Michael Ringel on the occasion of his 80th birthday}

\begin{abstract}
	In this paper, we introduce the monomial arrow removal operation for bound quiver algebras, and show that it is a novel reduction technique for determining the finiteness of the finitistic dimension. Our approach first develops a general method within the theory of abelian category cleft extensions. We then demonstrate that the specific conditions of this method are satisfied by the cleft extensions arising from strict monomial arrow removals. This crucial connection is established through the application of non-commutative Gr\"{o}bner bases in the sense of Green. The theory is illustrated with various concrete examples.
\end{abstract}

\maketitle

\setcounter{tocdepth}{1} \tableofcontents

\section{Introduction and main result}

The arrow removal operation was introduced in \cite{GPS} in order to investigate the Finitistic Dimension Conjecture (FDC). It was proved that this operation reduces the question of the finiteness of the finitistic dimension of a finite-dimensional algebra $\Lambda$ to that of its arrow removal, providing thus a reduction technique for the FDC. Recall that the finitistic dimension $\findim\Lambda$ is the supremum of the projective dimensions of all finitely generated right $\Lambda$-modules of finite projective dimension, and the FDC claims that it is finite for all $ \La $ as above. The arrow removal operation was investigated further in \cite{EPS} with respect to Gorensteinness, singularity categories and the finite generation condition for Hochschild cohomology. The key idea of the arrow removal operation is that the module category of a bound quiver algebra $ \La $ and its arrow removal algebra are connected via a cleft extension.

The ultimate goal of our approach is to understand how the arrows of a bound quiver algebra contribute to its finitistic dimension and then find methods to remove the ones that do not. This motivates the following problem, extending the scope of the arrow removal operation:

\begin{problem}
	Determine explicit combinatorial conditions so that an arrow in a relation can be removed establishing further reduction techniques for the FDC.
\end{problem}

In this paper, we continue our investigations toward the above problem. 
Let $\Lambda=kQ/I$ be an admissible quotient of a path algebra $kQ$, and $ \a $ an arrow in $Q$.
We introduce three different kinds of arrow removal based on the existence of a generating set $ T $ for $ I $ of the following types; see also \cref{monomialarrowrem}. We use $ p $ and $ q $ throughout to denote two paths not divided by the arrow $ \a $ such that $ p \a q $ is non-zero in $ k Q $.
	\begin{enumerate}[\rm(i)]
	\item The set $ T $ is \emph{$ \a $-monomial} if every path in the support of every element $ t \in T $ either equals $ t $ (i.e.\ $ t $ is a path) or does not contain the arrow $ \a $ as a subpath;
	\item The set $ T $ is \emph{single $ \a $-monomial} if there are paths $p $ and $ q $ as above such that:
	\begin{enumerate}[\rm(a)]
		\item the path $ p \a q $ is in $ T $, and no path occurring in an element of the set $ T \setminus \{ p \a q \} $ is divided by $ \a $;
		\item no proper subpath of  $p \a q $ is in $ I $;
		\item at most one of $p$ and $q$ is trivial (in order to exlude the arrow removal operation setup).
	\end{enumerate}
	\item The set $ T $ is \emph{strict $ \a $-monomial} if it is single $ \a $-monomial for paths $ p $ and $ q $ as above, and no path occurring in $ T $:
	\begin{enumerate}[\rm(a)]
		\item overlaps with $p$ from the right;
		\item overlaps with $q$ from the left;
		\item divides $p$ or $q$.
	\end{enumerate}
\end{enumerate}

Our main result provides a new reduction technique:

\begin{main-thm}[\mbox{\cref{maincor}{}}]
	Let $\Lambda=kQ/I$ be a bound quiver algebra. Let $ \a $ be an arrow in $Q$ such that $ I $ possesses a strict $ \a $-monomial generating set, and let $ \Ga = \La / \la \ol{ \a } \ra $ where $ \la \ol{ \a } \ra $ denotes the ideal of $ \La $ generated by $ \ol{ \a } = \a + I $. Then
	\[
	\findim\Lambda \leq \findim\Gamma + 2 .
	\]
In particular, it holds that $ \fpd \Ga < \infty $ implies $ \fpd \La < \infty $.
\end{main-thm}

The proof of the above theorem is an application of \cref{thm:findimcleftextab}, which is a general result for the finitistic dimension in the context of abelian category cleft extensions. A key assumption in \cref{thm:findimcleftextab} is the existence of common bounds for the projective dimensions of objects in certain functor images, which capture the homological complexity of the cleft extension.
Moreover, the cleft extension of module categories associated to a monomial arrow removal situation satisfies the aforementioned homological properties, a fact which we establish via non-commutative Gr\"{o}bner bases in the sense of Green \cite{Green2}{}; see \cref{subsec:gr.bases}.

It should also be noted that $\fpd \La < \infty $ implies $\fpd \Ga < \infty $ in the context of our main theorem, see \cite[Theorem~3.8]{Giata}{}. While the second author develops a more comprehensive reduction technique there, it is not combinatorially accessible in general due to the inherent homological nature of the assumptions.

The contents of the paper section by section are as follows. \cref{sec:mon.ar.rem} is divided into two subsections. In \cref{subsec:mon.ar.rem}{}, we introduce strict monomial arrow removal algebras and their basic properties. Furthermore, we prove that the monomial arrow removal operation induces a cleft extension of module categories, see \cref{propcleftextension}.
In \cref{subsec:gr.bases}{}, we introduce strict monomial Gr\"{o}bner bases, which provide the most general context for our reduction technique. Furthermore, we show that the existence of such a Gr\"{o}bner basis is guaranteed whenever we have a strict monomial arrow removal, see \cref{lem:1}{}. The converse is not true is general as we illustrate in \cref{exam:1''}{}.

In \cref{sectioncleftextabcat}, we briefly recall the abstract categorical framework of cleft extensions of abelian categories together with some useful homological properties, and prove \cref{thm:1}. In \cref{sectionproofofmainthm}, we prove the main result of this paper based on \cref{thm:1}. In particular, we explicitly calculate the minimal projective resolutions of certain key bimodules; see \cref{lem:2}{} and \cref{lem:3}{}. 
Finally, in \cref{sectionexamples}, we recall as \cref{defn:reduced} the class of reduced algebras introduced in \cite{GPS}, and finish the paper by showing the finiteness of the finitistic dimension of various concrete reduced algebras using our reduction technique.

\subsection*{Acknowledgments} 
The present research project was partly supported by the Hellenic Foundation for Research and Innovation (H.F.R.I.). Specifically, the second author received support under the ``3rd Call for H.F.R.I.\ Ph.D.\ Fellowships'' (FN: 47510/03.04.2022). For the third author, the research project was implemented in the framework of the H.F.R.I.\ call ``Basic Research Financing (Horizontal Support of all Sciences)" under the National Recovery and Resilience Plan ``Greece 2.0" funded by the European Union - NextGenerationEU (H.F.R.I.\ Project Number: 16785).

\section{Monomial arrow removal and Gr{\"o}bner bases}		\label{sec:mon.ar.rem}

Throughout the rest of the article we fix an admissible quotient $\La = k Q / I$ of the path algebra $kQ$, that is we fix a finite quiver $Q$, a field $k $ and an admissible ideal $I$ of $k Q$. Furthermore, we fix an arrow $ \a $.

In this section, we introduce the notion of a strict $ \a $-monomial arrow removal algebra of $ \La $ and prove that, if $\La / \la \ol{ \a } \ra $ is such an algebra, then there is a Gr{\"o}bner basis for $I$ satisfying some special properties that will be essential in the proof of our main theorem.

\subsection{Monomial arrow removal}		\label{subsec:mon.ar.rem}

In this subsection, we consider three conditions for a generating set of the admissible ideal $ I $ of $k Q$, and establish some consequences.

In what follows, we say that a path \emph{avoids} an arrow if it is not divided by it. Similarly, an element of $ k Q $ avoids an arrow if this is the case for every path occurring in the element. Even more generally, a subset of $ k Q $ avoids an arrow if this is the case for every element in it. Similarly, we say that a path occurs in some subset of $ k Q $ if it occurs in some element of the subset. Finally, a \emph{relation} is as usual a linear combination of paths of length at least two, where all paths have the same source and the same target.

\begin{defn}
	\label{defn:1}		\hypertarget{defn:1}{}
	Let $ T $ be a finite generating set of relations for the admissible ideal $ I $ of $k Q$. Let $ p $ and $ q $ be two paths avoiding the fixed arrow $ \a $ such that $t(p) = s ( \a )$ and $s(q ) = t( \a )$.
	\begin{enumerate}[\rm(i)]
		\item The set $ T $ is \emph{$ \a $-monomial} if every element in it that is not a path avoids $ \a $.
		\item The set $ T $ is \emph{single $ \a $-monomial} if there are paths $p $ and $ q $ as above such that:
			\begin{enumerate}[\rm(a)]
				\item $p \a q \in T $ and the set $ T \setminus \{ p \a q \} $ avoids $ \a $;
				\item no proper subpath of $p \a q $ is in the ideal generated by $ T $;
				\item at most one of $p$ and $q$ is trivial.
			\end{enumerate}
		\item The set $ T $ is \emph{strict $ \a $-monomial} if it is single $ \a $-monomial for paths $ p $ and $ q $ as above, and no path occurring in $ T $:
		\begin{enumerate}[\rm(a)]
			\item overlaps with $p$ from the right;
			\item overlaps with $q$ from the left;
			\item divides $p$ or $q$.		\phantomsection		\hypertarget{c}{}
		\end{enumerate}
	\end{enumerate}
	We write $ T = T_{ p , q }$ in cases (ii) and (iii) to indicate the paths $p$ and $q$.
\end{defn}

\begin{rem}
	If $s$, $t$ are two paths, then \emph{$ s $ overlaps with $ t $ from the left} or, equivalently, the path \emph{$ t $ overlaps with $ s $ from the right}, if there are paths $r_1$, $r_2$ and $r_3$ such that $r_2$ is non-trivial, such that $t = r_1 r_2$ and $s = r_2 r_3 $. Equivalently, the path $ s $ overlaps with $ t $ from the left if there are paths $u$, $v$ such that $t v = u s $, where $s $, $ t$ do not divide $v$ and $u$, respectively.
	Note that if one of the paths $p$ and $q$ in \cref{defn:1}{} is trivial, then the corresponding no-overlapping condition is automatically true as a trivial path does not overlap with any path.
\end{rem}

\begin{rem}
	The condition that $ T $ is finite and consists of relations is imposed for technical reasons only. Indeed, if $ T $ is any set generating an admissible ideal $I$ of $k Q$, then $T' = \{ e_i t e_j \colon i , j \in Q_0 , t \in T \} $ is also a generating set for $I$ consisting of relations. Since $I$ is finitely generated as an ideal of $k Q$ (see for instance \cite[Lemma II.2.8]{ASS}{}), there is a finite subset $ T'' \subseteq T'$ generating $I$. Moreover, if $T$ satisfies any of the conditions in \cref{defn:1}{}, then the same holds for $T''$.
\end{rem}

We are now ready to give the definition of a (strict, single) monomial arrow removal algebra.

\begin{defn}
	\label{monomialarrowrem}
	The quotient algebra $ \La / \la \ol{ \a } \ra $ is a \emph{(strict or single) monomial arrow removal of $\La $} if $I$ is generated by a (strict or single, respectively) $ \a $-monomial set.
\end{defn}

In the following example we provide an instance of a single monomial arrow removal algebra. An example of a strict monomial arrow removal is deferred to \cref{sectionexamples}{}.

\begin{exam}
	\label{exam:1}
	Let $Q$ be the quiver of the following figure. Let $ T $ be the set consisting of:
	\begin{enumerate}[\rm(i)]
		\item the paths $\theta \a \b \g_1 \d_1 $ and $ \e \z $, and $\g_2 \d_2 - \g_1 \d_1 $;
		\item the paths of the form $\l_i^{j_i}$ for integers $j_i \geq 2 $, where $i = 1, 2, \ldots, 9 $;
		\item all the paths of length two involving one loop and one non-loop arrow except for the paths $\l_1 \a $, $ \a \l_2 $, $\d_1 \l_6 $, $\d_2 \l_6 $ and $\l_9 \theta $.
	\end{enumerate}
	It is not difficult to see that $ \La = k Q / I $ is a bound quiver algebra, where $ I $ is the ideal of $ k Q $ generated by $ T $ and $ k $ is any field. Indeed, it suffices to see that the maximal paths in $ Q $ starting with $\theta$ that are not in $ I $ are the paths $\theta \a \l_2^{j_2 - 1}$, $\theta \a \b \g_1 $ and $\theta \a \b \g_2 $.

	\smallskip
	
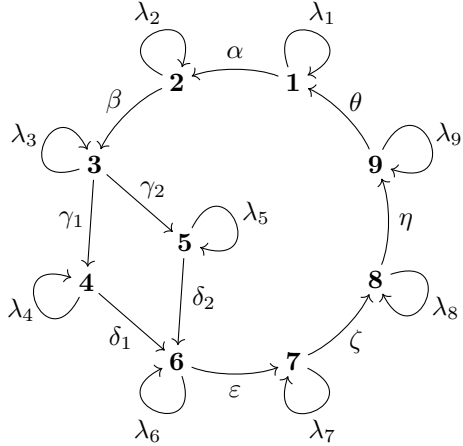
\begin{figure}[h!]
		\centering
		\begin{tikzpicture}	


			\node at ($(0,0)+(67.5:2)$) {$\mathbf{1}$};
			
			\node at ($(0,0)+(112.5:2)$) {$\mathbf{2}$};
			
			\node at ($(0,0)+(157.5:2)$) {$\mathbf{3}$};

			

			\path 
			($(0,0)+(157.5:2)$)			coordinate (3)      
			($(0,0)+(247.5:2)$)			coordinate (6)
			;
			
			\coordinate (Mid 3 6) at ($(3)!0.5!(6)$);

			\node at ($(Mid 3 6)+(202.5:0.7)$) {$\mathbf{4}$};
			
			\node at ($(Mid 3 6)+(22.5:0.7)$) {$\mathbf{5}$};

			\node at ($(0,0)+(247.5:2)$) {$\mathbf{6}$};
			
			\node at ($(0,0)+(292.5:2)$) {$\mathbf{7}$};
			
			\node at ($(0,0)+(337.5:2)$) {$\mathbf{8}$};
			
			\node at ($(0,0)+(22.5:2)$) {$\mathbf{9}$};


			\draw[->,shorten <=6pt, shorten >=6pt] ($(0,0)+(67.5:2)$) arc (67.5:112.5:2);
			
			\draw[->,shorten <=6pt, shorten >=6pt] ($(0,0)+(112.5:2)$) arc (112.5:157.5:2);

			\draw[->,shorten <=6pt, shorten >=6pt]  ($(0,0)+(157.5:2)$) -- ($(Mid 3 6)+(202.5:0.7)$);
			
			\draw[->,shorten <=6pt, shorten >=6pt]  ($(0,0)+(157.5:2)$) -- ($(Mid 3 6)+(22.5:0.7)$);

			\draw[->,shorten <=6pt, shorten >=6pt] ($(Mid 3 6)+(202.5:0.7)$) -- ($(0,0)+(247.5:2)$);
			
			\draw[->,shorten <=6pt, shorten >=6pt] ($(Mid 3 6)+(22.5:0.7)$)  -- ($(0,0)+(247.5:2)$);

			\draw[->,shorten <=6pt, shorten >=6pt] ($(0,0)+(247.5:2)$) arc (247.5:292.5:2);
			
			\draw[->,shorten <=6pt, shorten >=6pt] ($(0,0)+(292.5:2)$) arc (292.5:337.5:2);
			
			\draw[->,shorten <=6pt, shorten >=6pt] ($(0,0)+(337.5:2)$) arc (337.5:382.5:2);
			
			\draw[->,shorten <=6pt, shorten >=6pt] ($(0,0)+(22.5:2)$) arc  (22.5:67.5:2);


			\draw[->,shorten <=6pt, shorten >=6pt] ($(0,0)+(67.5:2)$).. controls +(67.5+41:1.2) and +(67.5-41:1.2) .. ($(0,0)+(67.5:2)$) ;
			\node at ($(0,0)+(67.5:3)$) {$\l_1$};
			
			\draw[->,shorten <=6pt, shorten >=6pt] ($(0,0)+(112.5:2)$).. controls +(112.5+41:1.2) and +(112.5-41:1.2) .. ($(0,0)+(112.5:2)$) ;
			\node at ($(0,0)+(112.5:3)$) {$\l_2$};
			
			\draw[->,shorten <=6pt, shorten >=6pt] ($(0,0)+(157.5:2)$).. controls +(157.5+41:1.2) and +(157.5-41:1.2) .. ($(0,0)+(157.5:2)$) ;
			\node at ($(0,0)+(157.5:3)$) {$\l_3$};

			\draw[->,shorten <=6pt, shorten >=6pt] ($(Mid 3 6)+(202.5:0.7)$).. controls +(202.5+41:1.2) and +(202.5-41:1.2) .. ($(Mid 3 6)+(202.5:0.7)$) ;
			\node at ($(0,0)+(202.5:3.05)$) {$\l_4$};
			
			\draw[->,shorten <=6pt, shorten >=6pt] ($(Mid 3 6)+(22.5:0.7)$).. controls +(22.5+41:1.2) and +(22.5-41:1.2) .. ($(Mid 3 6)+(22.5:0.7)$) ;
			\node at ($(0,0)+(22.5:0.3)$) {$\l_5$};

			\draw[->,shorten <=6pt, shorten >=6pt] ($(0,0)+(247.5:2)$).. controls +(247.5+41:1.2) and +(247.5-41:1.2) .. ($(0,0)+(247.5:2)$) ;
			\node at ($(0,0)+(247.5:3)$) {$\l_6$};
			
			\draw[->,shorten <=6pt, shorten >=6pt] ($(0,0)+(292.5:2)$).. controls +(292.5+41:1.2) and +(292.5-41:1.2) .. ($(0,0)+(292.5:2)$) ;
			\node at ($(0,0)+(292.5:3)$) {$\l_7$};
			
			\draw[->,shorten <=6pt, shorten >=6pt] ($(0,0)+(337.5:2)$).. controls +(337.5+41:1.2) and +(337.5-41:1.2) .. ($(0,0)+(337.5:2)$) ;
			\node at ($(0,0)+(337.5:3)$) {$\l_8$};
			
			\draw[->,shorten <=6pt, shorten >=6pt] ($(0,0)+(22.5:2)$).. controls +(22.5+41:1.3) and +(22.5-41:1.3) .. ($(0,0)+(22.5:2)$) ;
			\node at ($(0,0)+(22.5:3.05)$) {$\l_9$};


			\node at ($(0,0)+(90:2.25)$) {$\alpha$};
			
			\node at ($(0,0)+(135:2.25)$) {$\beta$};

			\coordinate (gamma 1) at ($(3)!0.5!($(Mid 3 6)+(202.5:0.7)$)$);
			
			\coordinate (gamma 2) at ($(3)!0.5!($(Mid 3 6)+(22.5:0.7)$)$);
			
			\coordinate (delta 1) at ($($(Mid 3 6)+(202.5:0.7)$)!0.5!(6)$);
			
			\coordinate (delta 2) at ($($(Mid 3 6)+(22.5:0.7)$)!0.5!(6)$);

			\node at ($(gamma 1)+(170:0.25)$) {$\gamma_1$};
			
			\node at ($(gamma 2)+(45:0.25)$) {$\gamma_2$};

			\node at ($(delta 1)+(235:0.25)$) {$\delta_1$};
			
			\node at ($(delta 2)+(10:0.31)$) {$\delta_2$};

			\node at ($(0,0)+(270:2.25)$) {$\varepsilon$};			
			
			\node at ($(0,0)+(315:2.25)$) {$\zeta$};
			
			\node at ($(0,0)+(0:2.25)$) {$\eta$};
			
			\node at ($(0,0)+(45:2.25)$) {$\theta$};

		\end{tikzpicture}	
	\end{figure}

	We begin by showing that the only possible paths $p$ and $q$ such that there is a single $\a$-monomial generating set $T = T_{p , q}$ of $ I $ are $p = \theta $ and $q = \b \g_1 \d_1 $ or $q = \b \g_2 \d_2 $. The maximal non-zero paths of $ \La $ that end at vertex $ 1 $ are the paths $\l_1 ^{ j_1 - 1}$, $\l_9 ^{ j_9 - 1} \theta $ and $ \z \h \th $. Therefore, path $p$ must divide one of the above paths from the right. Similarly, 
	the maximal non-zero paths of $\La_1$ starting at vertex 2 are the paths $\l_2 ^{ j_2 - 1 } $, $ \b \g_1 \d_1 \l_6 ^{ j_6 - 1 }$ and $ \b \g_1 \d_1 \e $, and $q$ has to divide one of them, or their variants where the subpath $\g_1 \d_1 $ is replaced by the path $\g_2 \d_2 $, from the left. If $p $ divides $ \l_1 ^{ j_1 - 1}$ and $q$ is any of the aforementioned candidate paths, then $p \a q \in I$ implies that $q \in I$, a contradiction. Therefore, path $p $ is divided by $ \th $ from the right. A similar argument shows that $q $ cannot be a subpath of $\l_2 ^{ j_2 -1 } $. Furthermore, if we take $ q $ to be $\b \g_1 $ or $\b \g_2 $ then, for any candidate path $p$, it holds that $p \a q \notin I$. We deduce that $q $ has to be divided by $\b \g_1 \d_1 $ or $\b \g_2 \d_2 $ from the left. Since $\th \a \b \g_1 \d_1 = \th \a \b \g_2 \d_2 \in I$, it holds that $ p = \th $ and $q = \b \g_1 \d_1 $ or $q = \b \g_2 \d_2 $.
	
	Note that the above arguments show that $ T = T_{p , q} $ is a single $ \a $-monomial generating set for $ I $, where $ p = \th $ and $ q = \b \g_1 \d_1 $. Indeed, it follows that no proper subpath of $ p \a q $ is in $ I $. However, the generating set $ T $ is not strict $ \a $-monomial as the path $ \g_1 \d_1 $ overlaps from the left with and divides $q$. Moreover, the ideal $ I $ possesses no strict $\a$-monomial generating set as the paths $ \g_1 \d_1 $ and $\g_2 \d_2 $ occur in every generating set for $I_1$, their length being $ 2 $. In particular, the quotient algebra $ \La / \la \ol{ \a } \ra $ is a single but not strict monomial arrow removal of $ \La $.
\end{exam}

Now consider the following properties for a subset $T$ of the path algebra $k Q$, which will be referred to collectively as properties \hyperlink{P1}{(P)}. Recall that $p $ and $ q $ denote two paths that avoid the fixed arrow $ \a $, and such that $t(p ) = s ( \a )$ and $ s( q ) = t ( \a )$.
Furthermore, note that we say that a path occurs in a subset of $ k Q $ if it occurs in some element of the subset.
 
\begin{enumerate}
	\item[(P1)\phantom{$^{\textrm{op}}$}]	\phantomsection	\hypertarget{P1}{}		The set $T \setminus \{ p \a q \} $ avoids $ \a $.
	\item[(P2)\phantom{$^{\textrm{op}}$}]	\phantomsection	\hypertarget{P2}{}		The path $ p $ does not overlap from the left with any path occurring in $T$.
	\item[(P2)$^{\textrm{op}}$]				\phantomsection	\hypertarget{P2op}{}	The path $ q $ does not overlap from the right with any path occurring in $T$.
	\item[(P3)\phantom{$^{\textrm{op}}$}]	\phantomsection	\hypertarget{P3}{}		No path occurring in $T$ divides $p $ or $ q $.
\end{enumerate}

An element $t \in k Q$ satisfies any of the above properties if $T = \{ t \} $ does, by definition. For instance, the element $t$ satisfies pro\-per\-ty \hyperlink{P1}{(P1)} if $t = p \a q $ or $t $ avoids $ \a $. In particular, the set $ T $ is allowed not to include the path $p \a q $ in the context of property \hyperlink{P1}{(P1)}. If it holds in addition that $ p \a q \in T $, then we say that $ T $ satisfies property \hyperlink{P1}{(P1)} \emph{non-trivially}. Similarly, we say that $ T $ satisfies properties \hyperlink{P1}{(P)} \emph{non-trivially} if it satisfies properties \hyperlink{P1}{(P)} and it satisfies property \hyperlink{P1}{(P1)} non-trivially.

We proceed by showing that requiring condition (b) of \cref{defn:1}{}.(ii) is redundant in the definition of a strict $\a$-monomial generating set in the same definition.

\begin{lem}
	\label{lem:0}
	A finite generating set of relations $ T $ for $ I $ is strict $ \a $-monomial if and only if it satisfies properties \hyperlink{P1}{\textup{(P)}} non-trivially.
\end{lem}

\begin{proof}
	We only have to show that no proper subpath of $p \a q $ is in $I$ if $ T = T_{ p , q } $ is a finite set of relations satisfying relations \hyperlink{P1}{(P)} non-trivially. Let $p = p_1 p_2 $ and $q = q_1 q_2$ where at least one of $p_2$ and $q_1$ is a proper subpath, and note that $p_2 \a q_1 $ cannot be generated by $p \a q$, i.e.\ the path $p_2 \a q_1 $ cannot be in the ideal of $ k Q $ generated by $ p \a q $. Therefore, if $p_2 \a q_1 $ is in $I$, then there is a path $t$ occurring in $ T \setminus \{ p \a q \} $ that divides $ p_2 \a q_1$. Since $ t $ avoids $ \a $ according to property \hyperlink{P1}{(P1)}, it holds that $t$ divides $p_2$ or $q_1$, a contradiction to property \hyperlink{P3}{(P3)}.
\end{proof}

We close this subsection by showing that, if $\La / \la \ol{ \a } \ra $ is a monomial arrow removal of $\La$, then the two algebras form a cleft extension. The following result, whose proof is included for the sake of completeness, is a particular case of various results found in \cite[Sections~3~and~4]{Giata}.

\begin{prop}
	\label{propcleftextension}
	If the quotient algebra $ \Ga = \Lambda/ \langle \overline{ \a } \rangle$ is a monomial arrow
	removal of the bound quiver algebra $\La = k Q / I $, then $ \Ga \simeq k Q^* / I^* $ where $Q^*$ is the quiver $Q$ with the arrow $ \a $
	removed and $I^* = kQ^* \cap I$. Moreover, the inclusion of $Q^*$ into $ Q$ induces an algebra monomorphism $\nu \colon
	\Gamma \monicc \Lambda$, the projection of $Q $ onto $ Q^*$ induces an algebra epimorphism $\pi \colon \Lambda \epic \Gamma$, and $ \pi \nu = \id_\Ga $.
\end{prop}

\begin{proof}
Consider the commutative diagram of $k$-algebra homomorphisms
	\[\xymatrix{
		kQ^* \ar@{^(->}[r]^{\nu^*} \ar@{->>}[d]^{p^*} & kQ\ar@{->>}[d]^p  \ar@{->>}[r]^{\pi^*}   & k Q^*  \ar@{->>}[d]^{p^*} \\
		k Q^* / I^*  \ar@{^(->}[r]^{\nu'} & k Q / I \ar@{->>}[r]^-{\pi'} & k Q^* / I^*  \notag
	}\]
	where $\nu^*$ and $\pi^*$ are induced by the inclusion of $Q^* $ into $Q$ and the projection of $Q $ onto $Q^*$, respectively, and $\nu' $ is induced by $\nu^*$ since $I^* \subseteq I$ by definition. We have to show that $\pi^* (I ) \subseteq I^*$ in order for $\pi' $ to be a well-defined algebra epimorphism.
	
	Let $S$ be an $ \a $-monomial generating set of relations for $I$ and note that the vector space $kQ$ decomposes as $\langle \a \rangle \oplus kQ^*$.
	The ideal
	generated by the monomial relations in $S$ passing through $ \a $ is contained
	in $\langle \a \rangle$.  On the other hand, if $S'$ is the subset of $S$ that contains all relations of $S$ that avoid $ \a $, then the ideal of $k Q$ generated by $S'$ is 
	given as
	\[(\langle \a \rangle \oplus { kQ^* } ) S' (\langle \a \rangle \oplus { kQ^* } ) = [ ( \langle \a \rangle  S' { k Q } ) + ( { k Q } S' \langle \a \rangle ) ] \oplus ( { kQ^* } S' { kQ^* } ) \]
	and, thus,  it is contained in the ideal $\langle \a \rangle + I^*$ of $ k Q $. In particular, the ideal $ \la \a \ra + I  $ of $k Q $ is contained in $ \la \a \ra + I^*$, and $\pi^* ( z ) \in I^*$ for every $z \in I$.
	We deduce that $\pi'$ is well-defined. Furthermore, the equality $\pi' \nu' = \id_{ k Q^* / I^*} $ follows immediately from the equality $\pi^* \nu^* = \id_{k Q^* }$. 
	
	Observe now that it is also true that the ideal $ \la \a \ra + I^* = \la \a \ra \oplus I^* $ is contained in $ \la \a \ra + I$ by the definition of $I^*$. Therefore we have the equality $\langle \a \rangle + I
	= \langle \a \rangle \oplus I^*$.  This
	implies in particular the isomorphisms
	\begin{align}
		\Ga = \Lambda/\langle \overline{ \a } \rangle & \simeq kQ/(\langle \a \rangle + I)\notag\\
		& \simeq (\langle \a \rangle \oplus kQ^*)/(\langle \a \rangle \oplus
		I^*)\notag\\
		& \simeq kQ^*/I^* \notag
	\end{align}
	as $k$-algebras. It remains to set $\nu = \nu' \phi $ and $ \pi = \phi^{-1} \pi' $ for an isomorphism $ \phi \colon \Ga \isomto k Q^* / I^* $.
\end{proof}

\subsection{Gr{\"o}bner bases}		\label{subsec:gr.bases}

In this subsection, we introduce the notion of a strict $\a$-monomial Gr{\"o}bner basis for $I$. Moreover, we prove that, if $\La / \la \ol{ \a } \ra $ is a strict monomial arrow removal of $\La = k Q / I$, then the ideal $I$ possesses such a Gr{\"o}bner basis with respect to any admissible order on the set of paths in $Q$.

We begin by recalling some basic definitions and facts about Gr{\"o}bner bases in the sense of Green \cite{Green2}{}. The interested reader may work through \cref{exam:1'}{} in order to familiarize themselves with Gr\"{o}bner bases. We have included a detailed account of all the necessary calculations for this purpose. 

For the rest of this subsection, we fix an admissible order ``$\preceq $'' on $\BQ$, the set of paths in $Q$. We recall that an order ``$\preceq $'' on $\BQ $ is called \emph{admissible} if it satisfies the following conditions, where $ p , q, r \in  \BQ $.
	\begin{enumerate}[\rm(i)]
		\item It is a well-order.
		\item For every $p \preceq q $, we have $ p r \preceq q r $ whenever both $ p r $ and $ q r $ are non-zero.
		\item For every $p \preceq q $, we have $ r p \preceq r q $ whenever both $ r p $ and $ r q $ are non-zero.
		\item If $ r $ is a subpath of $ p $, then $ r \preceq p $.
	\end{enumerate}
	See also \cite[Subsection~2.2.2]{Green2}{}.

For every non-zero element $z = \sum \mu_p \cdot p \in k Q$, where $p $ ranges over $\BQ$ and almost all coefficients $\mu_p \in k$ are zero, the \emph{tip} of $z$ is defined to be
\[
\tip (z ) = \max{}_\preceq \{ p \in \BQ \, | \, \mu_p \neq 0 \}
\]
and $\ctip (z) $ is the coefficient of $\tip (z)$ in $z $. Furthermore, if $ T $ is any subset of $k Q$, then $\tip ( T ) = \{ \tip ( z) \, | \, z \in T \setminus \{0  \}  \}$ and $\ntip( T ) = \BQ \setminus \tip ( T )$.
A key fact is that $\ntip (I) + I$ is a $k$-basis of $\La$, consisting of the cosets of paths that do not occur as the tip of an element in $I$. Therefore, for every element $z \in k Q$ there is a unique element $N ( z ) \in {}_k \la \ntip (I) \ra $ such that $z - N ( z ) \in I$, called the \emph{normal form of $z$}. It should also be noted that the set of paths $\ntip ( I )$ enjoys the special property of being closed under subpaths.

Recall that a subset $ \GG \subseteq I$ is a \emph{Gr{\"o}bner basis} for $I$ if, for every $z \in I$, it holds that $\tip (g) $ divides $\tip (z) $ for some $g \in \GG $; see \cite[Subsection~2.2.3]{Green2}{}. Equivalently, the set $\GG$ is a Gr{\"o}bner basis for $I$ if $\tip (\GG)$ and $\tip (I)$ generate the same ideal of $ k Q $. It follows that a Gr{\"o}bner basis for $I$ is in particular a generating set for $I$ as ``$\preceq $'' is a well-order on $\BQ$. As $\tip (I)$ generates an admissible ideal of $ k Q $, it holds that there is a unique (finite) subset $ T \subseteq  \tip (I)$ such that $T $ and $\tip ( I )$ generate the same ideal of $ k Q $, and no proper subpath of a path in $ T $ is in that ideal. It turns out that the \emph{reduced Gr{\"o}bner basis of $I$} is the finite set $\GG_{\textrm{red}} = \{ t - N( t )   \}_{ t \in T } $; see \cite[Subsection~2.3.1]{Green2}{}. We also note that $ \tip ( \GG_{\textrm{red}} )$ is contained in $ \tip ( \GG ) $ for every Gr\"{o}bner basis $\GG $ for $ I $.

We continue by recalling and reformulating the algorithm in \cite[Subsection 2.4.1]{Green2}{} for the construction of a Gr{\"o}bner basis for $I$ given a finite set of generating relations, and the Division Algorithm in \cite[Subsection 2.3.2]{Green2}{}. The former algorithm, which we call the \emph{Gr{\"o}bner Basis Algorithm} from now on, is based on the latter and the notion of overlapping relations from \cite[Subsection 2.3.3]{Green2}{}.

Given a finite set $X = \{ x_1 , x_2 , \ldots , x_n \} \subseteq k Q $ and an element $y \in k Q$, we may divide $y$ by $X$ to get a \emph{remainder} $r \in k Q$, which we denote by $ y \Rightarrow _X r$. In order for $r$ to be uniquely defined we need to consider $X$ as an ordered set and, whenever a path $u $ divides a path $v $ in more than one places, we always consider, for instance, the leftmost division. Then the Division Algorithm may be reformulated as follows.

\begin{reform}[Division Algorithm]
	\label{reformI}
	For $X$ and $y$ as above, we construct recursively two sequences of elements in $k Q$, $\{ y_j \}_{j \geq 0}$ and $\{ r_j \}_{j \geq 0}$, according to the following rules.
	\begin{enumerate}
		\item[Rule I:] \phantomsection	\hypertarget{ruleI}{}	We set $y_0 : = y$ and $r_0 := 0$.
		\item[Rule II:] \phantomsection	\hypertarget{ruleII}{}	If $y_j \neq 0$ and $\tip (y_j)$ is divided by $\tip ( x_i )$ for some $i \in \{ 1, 2, \ldots, \k \}$,	
		take $i$ to be the minimal such non-negative integer and let $\tip (y_j ) = c \tip (x_i) d $ be the leftmost division of $\tip (y_j)$ by $\tip ( x_i )$ for paths $c$, $d$.
		Then
		\[
		y_{j+1} := y_j - ( \ctip (y_j ) \ctip (x_i)^{-1} ) \cdot c x_i d     \,\,\, \textrm{ and }  \,\,\,      r_{j+1} := r_j.
		\]
		\item[Rule III:] \phantomsection	\hypertarget{ruleIII}{}	 If $y_j \neq 0$ and Rule II does not apply, then
		\[
		y_{j+1 } := y_j - \ctip (y_{j } ) \cdot \tip (y_j )     \,\,\, \textrm{ and }  \,\,\,    r_{j+1 } := r_j + \ctip (y_{j } ) \cdot \tip (y_j ).
		\]
		\item[Rule IV:] If $y_j = 0$, then $r = r_j$ is the remainder of the division of $y$ by $X$. The algorithm stops here.
	\end{enumerate}
	Note that the algorithm does stop after a finite number of steps, since $\tip (y _{j+1 } ) \prec \tip ( y_j )$ for every integer $j \geq 0$ such that the elements $y_j$ and $y_{j+1}$ are non-zero, and ``$\preceq $'' is a well-order.
\end{reform}

Assume now that $z, w \in k Q$ and there are paths $u$, $v $ such that $\tip(z ) v = u \tip (w )$ and $\tip (z ) $, $\tip (w )$ do not divide $u $ and $v$, respectively. In other words, we assume that $\tip (z ) $ overlaps with $\tip (w )$ from the right. Then the \emph{overlapping relation of $(z, w)$ by $(u, v)$} is
\[
\oo (z, w, u, v) : = \ctip ( z ) ^{-1} \cdot z v - \ctip ( w ) ^{-1} \cdot u w \in k Q.
\]

Let $X \subseteq k Q $ be any finite set of relations. Then we may construct a finite superset of relations $\tilde{ X }$ of $ X $ as follows. First, we fix an arbitrary linear order on the elements of $ X $. Then we define an element $y \in k Q $ to be in $\tilde{ X }$ if $ y \in X $, or there are $z, w  \in X$ and an overlapping relation $\oo ( z, w, u, v )$ such that $ \oo ( z, w, u, v ) \Rightarrow_X y $. Note that $\tilde{ X }$ is finite as there are only finitely many overlapping relations for any pair of elements, if any. It is also worth noting that, even though $\tilde{ X }$ is in general larger than $ X $, it is still contained in the ideal generated by $ X $.

The Gr{\"o}bner Basis Algorithm can now be reformulated as follows.

\begin{reform}[Gr{\"o}bner Basis Algorithm]
	\label{reformII}
	Let $ T $ be a finite generating set of relations for $I$. Let $ T_0 : = T $ be equipped with any linear order and set $ T_{j+1} : = \tilde{ T_j } $ for every $j \geq 0$, where the elements of $ T_{j+1} $ are ordered linearly as follows. The elements of $ T_j $ retain the inherited linear order, the elements in $ T_{j+1 } \setminus T_j $ are ordered linearly in an arbitrary way, and every element of $ T_{j+1 } \setminus T_j $ is bigger than any element of $ T_j $. Then \cite[Proposition 2.8]{Green2}{} and \cite[Proposition 2.10]{Green2} ensure that the sequence $\{ T_j \}_{j \geq 0}$ is eventually constant. Moreover, the finite set (of relations) at which this sequence stabilizes is a Gr{\"o}bner basis for $I$.
\end{reform}

	\begin{exam}
	\label{exam:1'}
	Let $ \La = k Q / I $ be the bound quiver algebra of \cref{exam:1}{}. We compute the Gr\"{o}bner basis for $ I $ consructed from the generating set $ T $ according to the Gr\"{o}bner Basis Algorithm, with respect to any admissible order on $ \BQ $.

	We begin by fixing any admissible order ``$\preceq_1 $'' on the set $ \BQ $ of paths in $Q$ satisfying $ \g_1 \d_1 \prec_1 \g_2 \d_2 $. Then $\tip ( T ) $ contains all the paths occurring in $ T $ except for the path $\g_1 \d_1 $. According to \cref{reformII}{}, the first step we have to take is to compute $\tilde{ T }$. Note that an overlapping relation between two paths is always zero and, thus, it suffices to only consider the overlapping relations of $\g_2 \d_2 - \g_1 \d_1 $ with the elements of $ T $. The only such non-zero relation is
	\[
	\oo(\l_3 \g_2  \, , \,  \g_2 \d_2 - \g_1 \d_1  \, , \,  \l_3  \, , \,  \d_2 ) = \l_3 \g_2 \cdot \d_2 - \l_3 \cdot (\g_2 \d_2 - \g_1 \d_1) = \l_3 \g_1 \d_1 .
	\]
	But $\l_3 \g_1 \d_1 $ is divided by $\l_3 \g_1 \in T $, implying that $ \l_3 \g_1 \d_1 \Rightarrow_{ T } 0 $. In particular, we have that $\tilde{ T } = T $ is a Gr{\"o}bner basis for $ I $. We note that $ T $ happens to be the reduced Gr\"{o}bner basis of $ I $ in this case, as $\g_1 \d_1 \in \ntip(I_1)$.

	If we fix now any admissible order ``$\preceq_2 $'' on $ \BQ $ such that $ \g_2 \d_2 \prec_2 \g_1 \d_1 $, then the Gr{\"o}bner basis of $ I $ constructed from $ T $ is $\tilde{ T } = T \cup \{ \th \a \b \g_2 \d_2  \}$. Moreover, the reduced Gr\"{o}bner basis for $ I $ is $ \tilde{ T } \setminus \{ \th \a \b \g_1 \d_1 \} $ in this case.
\end{exam}

The existence of a Gr{\"o}bner basis for $I$ of the following special type will be the most general context in which our main theorem holds. As in \cref{subsec:mon.ar.rem}{}, we use $p$ and $q$ to denote two paths that avoid the fixed arrow $ \a $, such that $t( p ) = s ( \a )$  and $s( q) = t ( \a )$. Furthermore, we allow at most one of $p$ and $q$ to be trivial.

\begin{defn}
	\label{defn:2}
	A Gr{\"o}bner basis $\GG $ for $I$ will be called \emph{strict $ \a $-monomial} if there are paths $p$ and $q$ as above such that:
		\begin{enumerate}[\rm(i)]
			\item the set $ \GG $ satisfies property \hyperlink{P1}{(P1)} non-trivially, i.e.\ it contains the path $p \a q $;
			\item the set $\tip ( \GG )$ satisfies properties \hyperlink{P2}{(P2)}, \hyperlink{P2op}{(P2)$^{\textrm{op}}$} and \hyperlink{P3}{(P3)}.
		\end{enumerate}
	We write $\GG = \GG_{p , q} $ in order to indicate the paths $p$ and $q$.
\end{defn}

\begin{rem}
	\label{rem:1}
	If the admissible ideal $I$ possesses a strict $ \a $-monomial Gr{\"o}bner basis $ \GG = \GG_{p , q} $, then $ \a $ cannot be a loop. Indeed, if $ \a $ was a loop then there would be a positive integer $ i $ such that $ \a ^i \in I$, implying that $\tip ( g )$ divides $ \a ^i$ for some $ g \in \GG $. However, this is not possible as either $\tip ( g ) = p \a q$ or $\tip ( g ) $ avoids $ \a $.
\end{rem}

\begin{lem}
	\label{lem:00}
	The ideal $I$ possesses a strict $ \a $-monomial Gr{\"o}bner basis if and only if the reduced Gr{\"o}bner basis of $I$ is strict $ \a $-monomial. In particular, if $ \GG_{ p , q } $ and $ \GG'_{p' , q'} $ are two strict $ \a $-monomial Gr\"{o}bner bases for $ I $, then $ p ' = p $ and $ q' = q $.
\end{lem}

\begin{proof}
	Let $\GG = \GG_{p , q}$ be a strict $ \a $-monomial Gr{\"o}bner basis for $I$. Then $\GG$ is in particular an $ \a $-monomial generating set for $I$. For every element $ z \in k Q $, we use $ z_\a $ and $ z'$ to denote the unique elements of $k Q$ such that $z' $ avoids $ \a $, every path occurring in $z_\a $ is divided by $ \a $,  and $z = z_\a + z'$. Then it follows from the proof of \cref{propcleftextension}{} that $z_\a, z' \in I$ whenever $z \in I$.
	
	Let now $\GG_{\textrm{red}} $ be the reduced Gr{\"o}bner basis of $I$, and observe that $\tip (\GG_{\textrm{red}} )$ satisfies properties \hyperlink{P2}{(P2)}, \hyperlink{P2op}{(P2)$^{\textrm{op}}$} and \hyperlink{P3}{(P3)}, as $\tip (\GG_{\textrm{red}} ) \subseteq \tip (\GG)$. It remains to show that $p \a q \in \GG_{\textrm{red}} $ and $\GG_{\textrm{red}} \setminus \{ p \a q \} $ avoids $ \a $.
	
	Let $T$ denote the unique minimal generating set of the ideal of $ k Q $ generated by $ \tip (I) $, and observe that $ T \subseteq \tip ( \GG ) $. We begin by showing that $p \a q \in T$. If we assume that $p \a q \notin T$, then there is an element $ g \in \GG \setminus \{ p \a q \}$ such that $\tip ( g ) $ divides $ p \a q $. Since $ g $ avoids $ \a $, the path $\tip ( g ) $ divides $p$ or $q$, a contradiction as $\tip (\GG)$ satisfies property \hyperlink{P3}{(P3)}. Of course, we have $N( p \a q ) = 0$ as $p \a q \in I$. It remains to prove that for every path $t \in T \setminus \{ p \a q \}$ it holds that $t - N(t)$ avoids $ \a $. It is immediate that $t $ avoids $ \a $ as $ t \in \tip ( \GG  \setminus \{ p \a q \} )$. On the other hand, we have that $ { { ( t - N( t ) )' } = { t - N(t)' } } , { { ( t - N( t ) )_\a } = { N(t)_\a} } \in I$ since $t - N ( t ) \in I$. We conclude that $N(t)_\a = 0$, as the intersection of $I$ and ${}_k \la \ntip (I) \ra$ is trivial, and thus the element $N(t) = N(t)'$ avoids $ \a $.
	
	Now observe that, if $\GG' = \GG'_{p' , q'}$ is some other strict $ \a $-monomial Gr{\"o}bner basis for $I$, then $p' = p $ and $q' = q $ as both $p \a q $ and $p' \a q'$ are equal to the unique path of $T$ that is divided by $ \a $. It remains to note that $T$ is uniquely determined by $I$ and the admissible order on $\BQ$.
\end{proof}

We prove next that an admissible ideal $I$ admits a strict $ \a $-monomial Gr{\"o}bner basis whenever $\La / \la \ol{ \a } \ra $ is a strict monomial arrow removal of $ \La = k Q / I $. The converse is not true, however, as illustrated in \cref{exam:1''}{}.

\begin{lem}
		\label{lem:1}
	If the quotient algebra $\La / \la \ol{ \a } \ra $ is a strict monomial arrow removal of $\La = k Q / I $, then
	there is a finite Gr{\"o}bner basis for $I$ that is strict $ \a $-monomial (even as a generating set for $I$).
\end{lem}

\begin{proof}
	Let $ T = T_{p , q}$ be a finite strict $ \a $-monomial generating set of relations for $I$. According to \cref{lem:0}{}, this is equivalent to $ p \a q \in T $ and $ T $ satisfying properties \hyperlink{P1}{(P)} for the paths $p$ and $q$. It suffices to prove that if $X$ is any finite set of relations satisfying properties \hyperlink{P1}{(P)}, then the same holds for $\tilde{ X }$. Indeed, then the Gr\"{o}bner basis $ \GG $ constructed from $ T $ as in \cref{reformII}{} will be strict $ \a $-monomial as a generating set for $ I $, as it will satisfy properties \hyperlink{P1}{(P)} and $ p \a q \in T \subseteq \GG $.

	Let $ y = \oo ( z, w, u, v )$ be an overlapping relation for $z, w \in X$, and let $y \Rightarrow_X r $. The only thing we need to prove is that the remainder $r $ satisfies properties \hyperlink{P1}{(P)}.
	
	\smallskip
	
	\phantomsection	\hypertarget{stepI}{}	\textbf{Step I}: It holds that $z $ and $ w $ avoid $ \a $.
	
	We assume that $p \a q \in X$ because, otherwise, there is nothing to prove. Since the set $ X $ satisfies in particular property \hyperlink{P1}{(P1)}, we have to show that $ z, w \neq p \a q $.
	
	Let $z \neq p \a q $ and $w = p \a q $, and the case where $z = p \a q $ and $w \neq p \a  q$ is analogous. It holds that $\tip (z ) v = u p \a q $, where $ p \a q $ does not divide $v$ and $\tip (z ) $ does not divide $u$. On the other hand, the path $\tip (z) $ avoids $ \a $ implying that $\tip ( z) p_2 = u p $, where $p_2 $ is a proper subpath of $p $ such that $p = p_1 p_2$ and $v = p_2 \a q $. It follows that $p$ overlaps with $\tip (z )$ from the left as the paths $\tip (z)$ and $p$ do not divide $u$ and $p_2$, respectively, and this is a contradiction to property \hyperlink{P2}{(P2)} for the set $X$.
	
	Now let $z= p \a q = w $. Then $p \a q v = u p \a q $ where $ p \a q $ does not divide $u$ or $v$. We may assume that none of $u$ and $v$ is trivial because, otherwise, we would have $y = 0$ and there would be nothing to prove. It follows that the path $ v $ is divided by $ \a $.
	Consequently, we have $v = p_2 \a q $ for a proper subpath $p_2 $ of $p$ such that $p = p_1 p_2 $. In particular, we have $ ( p \a q ) p_2 = u p $ where $p $ does not divide $ p_2 $ and $ p \a q $ does not divide $ u $, implying that $ p $ overlaps from the left with $ p \a q $, a contradiction to property \hyperlink{P2}{(P2)} for the set $ X $.
	
	\smallskip
	
	\phantomsection	\hypertarget{stepII}{}	\textbf{Step II:} It holds that $y  $ satisfies properties (P) and $y \neq p a q$.
	
	Recall that
	\[
	y = \ctip ( z )^{-1} \cdot z v - \ctip ( w )^{-1} \cdot u w .
	\]
	Therefore, every path occurring in $y$ is of the form $ s v $ or $ u t $ for paths $ s $ and $ t $ occurring in $z$ and $w$, respectively. More generally, the paths $ s $ and $ t $ occur in $X$. It suffices to prove that the paths of the form $ s v $ as above satisfy properties \hyperlink{P1}{(P)} and avoid $ \a $, and the proof for paths of the form $ u t $ is analogous. We assume that $ v $ is non-trivial as, otherwise, there is nothing to prove.
	
	\hyperlink{P1}{(P1)} The equality $\tip (z) v = u \tip (w )  $, where the paths $\tip (w )$ and $\tip (z) $ do not divide $v$ and $u$, respectively, implies that the paths $u$ and $v$ avoid $ \a $. Indeed, this follows from property \hyperlink{P1}{(P1)} for $X$ and the fact that $z$ and $w$ avoid $ \a $ (\hyperlink{stepI}{Step I}{}). In particular, the paths of the form $ s v $  avoid $ \a $.
	
	\hyperlink{P2}{(P2)} We assume now that $p$ overlaps with a path $ s v $ from the left. Then either $p$ overlaps with $ v $ from the left, implying that $p$ overlaps with $\tip (w )$ from the left, or $p$ overlaps with $ s $ from the left. In both cases, we have reachead a contradiction as both $ \tip (w) $ and $ s $ occur in $X $.
	
	\hyperlink{P2op}{(P2)$^{\textrm{op}}$} If $q$ overlaps with $ s v $ from the right, then either $q$ overlaps with $ s $ from the right or $ s $ divides $q$. This is a contradiction in either case, as $X$ satisfies properties \hyperlink{P2op}{(P2)$^{\textrm{op}}$} and \hyperlink{P3}{(P3)}.

	\hyperlink{P3}{(P3)} Assume that a path of the form $ s v $ divides $p $ or $q$. Then $ s  $ divides $p $ or $ q$, a contradiction as $X$ satisfies property \hyperlink{P3}{(P3)}.
	
	\smallskip
	
	\phantomsection	\hypertarget{stepIII}{}	\textbf{Step III:} Consider elements $f , g \in k Q$ such that $g \neq p a q$, the set $\{ f, g\}$ satisfies properties \hyperlink{P1}{(P)}, and $\tip ( g ) = c \tip ( f ) d $ for paths $c$ and $d$. Then the element
	\[
	h = g - ( \ctip (g ) \ctip ( f )^{-1} ) \cdot c f d
	\]
	satisfies properties \hyperlink{P1}{(P)} and $h \neq p a q $.

	\hyperlink{P1}{(P1)} Since $ g \neq p \a q  $ and the set $ \{ f , g  \} $ satisfies property \hyperlink{P1}{(P1)}, we have that $ g $ avoids $ \a $. Now the equation $\tip ( g ) = c \tip ( f ) d $ implies that the pahts $ \tip ( f ) $, $c$ and $d$ avoid $ \a $ as the same holds for $\tip ( g )$. In particular, we have that $ f \neq p \a q $ and therefore $ f $ avoids $ \a $ as the set $ \{  f , g \} $ satisfies property \hyperlink{P1}{(P1)}.
	
	\hyperlink{P2}{(P2)} If $p$ overlaps with $ h $ from the left, then $p$ overlaps with $c s d $ from the left for some path $ s $ occurring in $ f $. There are three cases. If $p$ overlaps with $ d $ from the left, then $p$ overlaps with $ \tip ( g ) $ from the left. If $p$ does not overlap with $d$ from the left but it does so with $ s d $, then $p$ overlaps with $ s $ from the left. If $p$ does not overlap with $ d $ or $ s d $ from the left then it does so with $c s d $, in which case $s$ divides $ p $. All three cases contradict the assumption that the set $ \{ f ,  g \} $ satisfies properties \hyperlink{P2}{(P2)} and \hyperlink{P3}{(P3)}.
	We omit the proof of property \hyperlink{P2op}{(P2)$^{\textrm{op}}$} as it is analogous to the proof of property \hyperlink{P2}{(P2)}.
	
	\hyperlink{P3}{(P3)} If a path occurring in $ h $ divides $ p $ or $ q $, then it has to be of the form $c s d $ for some path $ s $ occurring in $ f $. But then $s $ divides $p$ or $q$, a contradiction to the fact that the set $ \{ f ,  g \} $ satisfies property \hyperlink{P3}{(P3)}.
	
	\smallskip
	
	A careful examination of \cref{reformI}{} shows that $r$ satisfies properties \hyperlink{P1}{(P)} and is not equal to $ p \a q $ according to \hyperlink{stepII}{Step II}{} and \hyperlink{stepIII}{Step III}{}. Specifically, we have that the element $y_0 = y $ satisfies properties \hyperlink{P1}{(P)} and is not equal to $ p \a q $ according to \hyperlink{stepII}{Step II}{}. Assume now that the elements $y_j$ and $y_{j+1}$ are non-zero for some non-zero integer $ j $, where $y_j $ satisfies properties \hyperlink{P1}{(P)} and is not equal to $ p \a q $. If $y_j$ satisfies \hyperlink{ruleII}{Rule II}, then \hyperlink{stepIII}{Step III}{} implies that $ y_{j+1} $ satisfies properties \hyperlink{P1}{(P)} and is not equal to $ p \a q $. If $y_j$ satisfies \hyperlink{ruleIII}{Rule III}, then we reach the same conclusion trivially. The lemma follows now from the fact that every path occurring in $r$ occurs in $y_j$ for some non-negative integer $ j $.
\end{proof}

	\begin{exam}
		\label{exam:1''}
	Let $ \La = k Q / I $ be the bound quiver algebra of \cref{exam:1}{}. We show that $ I $ possesses a strict $ \a $-monomial Gr\"{o}bner basis with respect to any admissible order on $ \BQ $, even though the quotient algebra $ \La / \la \ol{ \a } \ra $ is not a strict monomial arrow removal of $ \La $ according to \cref{exam:1'}{}.
	
	Recall that $ T = T_{p, q}$ is the reduced Gr\"{o}bner basis for $ I $ with respect to any admissible order on $ \BQ $ such that $ \g_1 \d_1 $ is smaller than $ \g_2 \d_2 $. It is easy to see that $ T = T_{p, q}$ is strict $\a$-monomial as a Gr{\"o}bner basis by verifying property \hyperlink{P1}{(P1)} for $ T $ and properties \hyperlink{P2}{(P2)}, \hyperlink{P2op}{(P2)$^{\textrm{op}}$} and \hyperlink{P3}{(P3)} for $ \tip( T ) $. It should be noted that $ T $ is not strict $ \a $-monomial as a generating set for $ I $.

	On the other hand, if the fixed admissible order on $ \BQ $ is such that $ \g_2 \d_2 $ is smaller than $ \g_1 \d_1 $, then the Gr\"{o}bner basis $ \tilde{ T } = T \cup \{ \th \a \b \g_2 \d_2  \}$ constructed from $ T $ is not strict $ \a $-monomial. However, the reduced Gr\"{o}bner basis $\tilde{ T } \setminus \{ \th \a \b \g_1 \d_1 \} $ is strict $\a $-monomial, which is straightforward to verify.
\end{exam}

We close this section by noting that a loop can never be removed through a strict monomial arrow removal.

\begin{cor}
	If $ \a $ is a loop, the quotient algebra $\La / \la \ol{ \a } \ra $ is never a strict monomial arrow removal of $\La = k Q / I$.
\end{cor}

\begin{proof}
	Direct application of \cref{lem:1}{} and \cref{rem:1}{}.
\end{proof}

\section{Cleft extensions and the finitistic dimension}
\label{sectioncleftextabcat}

In this section, we establish conditions on an abelian category cleft extension that allow us to relate the finitistic dimensions of the involved categories.
Abelian category cleft extensions were introduced by Beligiannis \cite{Beligiannis}, and they
generalize abelian category trivial extensions due to
Fossum-Griffith-Reiten \cite{FosGriRei}.  We start by recalling basic results about them from \cite{Beligiannis,GPS}.

\begin{defn}\textnormal{(\!\!\cite[Definition~2.1]{Beligiannis})}
	\label{defncleftext}
	A \emph{cleft extension} of an abelian category $\B$ is an
	abelian category $\A$ together with functors:
	\[
	\xymatrix@C=0.5cm{
		\B \ar[rrr]^{\mi} &&& \A \ar[rrr]^{\me} &&& \B
		\ar @/_1.5pc/[lll]_{\ml} } 
	\]
	henceforth denoted by $ \C = (\B,\A, \me, \ml, \mi)$, such that the
	following conditions hold:
	\begin{enumerate}[\rm(a)]
		\item The functor $\me$ is faithful exact.
		
		\item The pair of functors $(\ml,\me)$ is adjoint.

		\item There is a natural isomorphism of functors
		$\varphi\colon \me \mi \isomto \mathsf{Id}_{\B}$.
	\end{enumerate} 
\end{defn}

For the rest of the section, we fix an abelian category cleft extension $ \C = (\B,\A, \me, \ml, \mi)$ as above.
We collect next some basic properties for the functors of $ \C $, most of which follow directly from the definition; see also \cite{Beligiannis}
and \cite[Lemma 2.2]{GPS}.

\begin{lem}
	\label{basicproperties}
	The following hold for the cleft extension $ \C = (\B,\A, \me, \ml, \mi)$.
	\begin{enumerate}[\rm(i)]
		\item The functor $\me\colon \A\lxr \B$ is essentially surjective. 
		
		\item The functor $\mi\colon \B \lxr \A$ is fully faithful and exact.
		
		\item The functor $\ml\colon \B\lxr \A$ is faithful and preserves
		projective objects.
		
		\item There is a functor $\mq\colon \A\lxr \B$ such that $(\mq,\mi)$
		is an adjoint pair.
		
		\item There is a natural isomorphism of
		functors $\mq\ml\simeq \mathsf{Id}_{\B}$.
		\item  The functor $\mq\colon \A \to \B$ preserves projective
		objects.
	\end{enumerate}
\end{lem}

We continue by recalling three fundamental functors associated to $ \C $, and their accompanying short exact sequences. In what follows, we write $ \eta \colon 1_\B\lxr \me\ml$ and $ \varepsilon \colon \ml\me\lxr 1_\A$ to denote the unit and counit, respectively, of the adjoint pair $(\ml,\me)$. Recall that the equality
$ 1_{\me(A)} = \me(\varepsilon_A) \eta_{\me(A)} $
holds for all $A \in \A$ as part of the definition of the adjunction.

The first functor is endofunctor $G\colon \A\lxr \A$ defined through the assignment $A\mapsto \Ker\varepsilon_A$. It comes along with the short exact sequence of functors
\phantomsection	\hypertarget{eq:G}{}
\begin{equation}		
		\label{eq:G}
	0\to G \xrightarrow{ \phantom{  \varepsilon  } } \ml\me \xrightarrow{ \, \varepsilon \, } \iden_{\A}\to 0 .
\end{equation}
To see that the morphism $\varepsilon_A \colon \ml \me ( A ) \to A $ is an epimorphism for every $ A \in \A $, note that morphism $\me(\varepsilon_A)$ is a split epimorphism (with section $ \eta_{\me(A)} $), and our claim follows from the fact that $\me$ is
faithful. 

The second functor is $H \colon \B \to \A $ defined through the assignment $B\mapsto G \mi(B) $, and the third one is the endofunctor $F \colon \B\lxr \B$ given by $ F = \me H $. Functor $ H $ comes along with the short exact sequence
\phantomsection	\hypertarget{eq:H}{}
\begin{equation}
		\label{eq:H}
	0 \to H \xrightarrow{ \phantom{ \, \varepsilon } }  \ml \xrightarrow{ \, \varepsilon \mi \, } \mi \to 0
\end{equation}
obtained through the application of the short exact sequence \hyperlink{eq:G}{(\ref{eq:G}{})}{} to the object $ \mi (B)$, taking into consideration the natural isomorphism $\varphi\colon \me \mi \isomto \mathsf{Id}_{\B}$. Similarly, functor $ F $ comes along with the short exact sequence
\phantomsection	\hypertarget{eq:F}{}
\begin{equation}
		\label{eq:F}
	0 \to F \xrightarrow{ \phantom{  \varepsilon \mi } } \me \ml \xrightarrow{ \, \me \varepsilon \mi \, } \iden_{\B}\to 0
\end{equation}
which is obtained from the short exact sequence \hyperlink{eq:H}{(\ref{eq:H}{})}{} by applying the exact functor $ \me $ and the identification $\varphi\colon \me \mi \isomto \mathsf{Id}_{\B}$.
Furthermore, the exact sequence \hyperlink{eq:F}{(\ref{eq:F}{})}{} splits as $ \eta $ is a section of $ \me \varepsilon \mi $; see also \cite[Lemma~2.3]{GPS}.

From now on, all abelian categories are assumed to have enough projective objects. In particular, the notions of projective dimension for an object and of finitistic dimension for the whole category are well-defined.

We establish next conditions on functors $\me$, $\mi$,
$H$ and $G$ associated to the abelian category cleft
extension $ \C $, so that $ \fpd \B < \infty $ implies $ \fpd \A < \infty $. Toward this end, we need the following quantities.

\begin{defn}
	For an abelian category cleft extension $ \C = (\B,\A, \me, \ml, \mi)$, we define the following quantities.
		\begin{enumerate}[\rm(i)]
			\item $ p_\A = \sup\{ \pd_\B \me(P)\mid P\in\Proj(\A)\} $; 
			\item $ p_\B = \sup\{ \pd_\A \mi(P')\mid P'\in\Proj(\B)\} $; 
			\item $ n_H = \sup\{ \pd_\A H(B) \mid B\in\B\} $;
			\item $ n_G = \sup\{ \pd_\A G(A) \mid A\in\A\} $. 
		\end{enumerate}
\end{defn}

\begin{lem}		\hypertarget{thm:findimcleftextab}{}
		\label{thm:findimcleftextab}
	For an abelian
	category cleft extension $ \C = (\B,\A, \me, \ml, \mi)$, the following inequalities hold for every $ A \in \A $ and $ B \in \B $.
	\begin{enumerate}[\rm(i)]
		\item	$ \pd_\B \me(A) \leq \pd_\A A + p_\A $; 
		\item	$ \pd_\A \mi (B ) \leq \pd_\B B + p_\B $;
		\item	$ \pd_\A \ml(B) \leq \max\{ n_H, \pd_\B B + p_\B\} $.
	\end{enumerate}
\end{lem}

\begin{proof}
	(i) If one of $ \pd_\A A $ and $ p_\A $ is infinite, then there is nothing to prove. Therefore, we assume that both are finite. In particular, there is a projective resolution of $ A $ of finite length equal to $ \pd_\A A $. Applying the exact functor $ \me $ to this resolution yields an exact sequence in $ \B $ of length $ \pd_\A A $, where all object except possibly for $ \me ( A )$ have projective dimension bounded above by $ p _\A $. The desired inequality follows now from a well-known homological argument.

	(ii) Similar to the proof of part (i).
	
	(iii) We assume that $ p_\B$, $ n _H $ and $ \pd_\B B $ are finite as, otherwise, there is nothing to prove. From sequence \hyperlink{eq:H}{(\ref{eq:H}{})}{} applied to $ B $ and a well-known homological argument, we have that
		\[
			\pd_\A \ml(B) \leq \max\{ \pd_\A H(B), \pd_\A \mi(B) \}
		\]
	and it remains to apply the inequality of part (ii).
\end{proof}

We prove next the main theorem of this section.

\begin{thm}	
		\label{thm:1}
	Let $ \C = (\B,\A, \me, \ml, \mi)$ be an abelian
	category cleft extension such that all the quantities $ p_\A $, $ p_\B $, $ n_H $ and $ n_G $ are finite. Then
		\[
			\fpd \B < \infty \, \implies \, \fpd \A < \infty .
		\]
\end{thm}

\begin{proof}
	Let $ A $ be an object of $ \A $ such that $ \pd_\A A < \infty $. From sequence \hyperlink{eq:G}{(\ref{eq:G}{})}{} applied to $ A $ and a well-known homological argument, we have that
		\[
			\pd_\A A \leq \max\{\pd_\A G(A), \pd_\A \ml\me(A)\} + 1 \leq \max\{ n_G , n_H , \pd_\B \me(A) + p_\B\} + 1
		\]
	where the second inequality follows from \hyperlink{thm:findimcleftextab}{\cref{thm:findimcleftextab}{}.(iii)}{}. It remains to note that $ \pd_\B \me(A) $ is bounded above by $ \fpd \B $ as it is finite due to \hyperlink{thm:findimcleftextab}{\cref{thm:findimcleftextab}{}.(i)}{}.
\end{proof}

\begin{rem}
	It holds that $ p_\B \leq n_H + 1 $.
	Indeed, sequence \hyperlink{eq:H}{(\ref{eq:H}{})}{} applied to any object $ P' \in \Proj (\B ) $ yields that $\pd_\A \mi(P') \leq \pd_\A H( P' ) + 1 \leq n_H + 1$, as functor $ \ml $ preserves projectivity according to \cref{basicproperties}{}. In particular, we obtain the (simplified) upper bound
 		\[
 			\fpd \A \leq \max \{ n_G , \fpd B + n_H + 1 \} + 1
 		\]
 	from the proof of \cref{thm:1}{}.
\end{rem}

	We close this section by describing functors $ G $, $ H $ and $ F $ for the special instance of ring cleft extensions. Recall that a \emph{ring cleft extension} is a quadruple $(\La, \Ga, \nu , \pi )$ where $ \nu \colon \Ga \monicc \La$ and $\pi \colon \La \epic \Ga$ are two ring homomorphisms such that $\pi \nu = \id_\Ga$, following the terminology of Beligiannis \cite{Beligiannis}{}.
	In particular, we have that $\nu$ is injective and $\pi$ is surjective, and we may identify $\Gamma$ with
	$\Lambda/ K $ where $ K = { \Ker \pi } $. For concrete examples of ring cleft extensions in addition to the ones in \cref{sectionexamples}{}, we refer the reader to \cite{GPS,EPS,Giata}.
	
	In the situation of a ring cleft extension as above, the diagram
	\begin{equation*}
		\label{digramwithendofunctors}
		\xymatrix@C=0.5cm{
			\sMod \Ga \ar[rrr]^{\mi = { ( - )_\La } } &&& \sMod \La \ar[rrr]^{\me = { ( - )_\Ga } } \ar @/_1.5pc/[lll]_{ \mq = { - \otimes_\La \Ga _\Ga } } &&& \sMod \Ga
			\ar @/_1.5pc/[lll]_{\ml = { - \otimes_\Ga \La _\La } } } 
	\end{equation*}
	is an abelian category cleft extension, where $ \me \mi = \iden_{ \sMod \Ga } $. Moreover, if $ \Ga $ and $ \La $ are finite-dimensional algebras then all the functors of the above diagram restrict to functors between the categories of finitely generated modules over $ \Ga $ and $ \La $, and the resulting diagram is again an abelian category cleft extension.
	
	It can be directly computed that functors $ F $ and $ H $ are naturally isomorphic to the tensor functors induced by the ideal $ K = \Ker \pi $ viewed as a $ \Ga $-bimodule and as a $ \Ga $-$ \La $-bimodule, respectively. We only recall here that the standard counit of the adjoint pair $ ( \ml , \me ) $ in the context of ring cleft extensions is $ \varepsilon_Y \colon Y \otimes_\Ga \La \epic Y $ given by $ y \otimes a \mapsto y a $ for every $ y \in Y $ and $ a \in \La $, for every $ \La $-module $ Y $.
	Furthermore, the $ \Ga $-actions on $ K $ are induced by the restriction of scalars $ \me = { ( - )_\Ga } $. For the rest of the paper, we make the identifications $ F = { - \otimes_\Ga K_\Ga } $ and $ H = { - \otimes_\Ga K_\La } $.
	
	The calculation of functor $ G $ as a tensor product requires a more detailed treatment. In what follows, we write $ L $ to denote the kernel of $ \varepsilon_{\La} \colon \La \otimes_\Ga \La \epic \La $, that is $ L = G ( \La_\La ) $ is the kernel of the map $ \La \otimes_\Ga \La \epic \La $ induced by multiplication in $ \La $.

	\begin{lem}
		For a ring cleft extension $(\La, \Ga, \nu , \pi )$, functor $ G $ is naturally isomorphic to the tensor functor induced by the kernel of the multiplication map $ \La \otimes_\Ga \La \epic \La $.
	\end{lem}

	\begin{proof}
	Observe that $ \varepsilon_{ \La } $ is a homomorphism of $ \La $-bimodules, and therefore $ L $ has a natural $ \La $-bimodule structure inherited by the natural $ \La $-bimodule structure of $ \La \otimes_\Ga \La $.
	Tensoring the short exact sequence $ 0 \to L \to { \La \otimes_\Ga \La } \xrightarrow{ \, \varepsilon_\La \, } \La \to 0 $ of $ \La $-bimodules with any module $ Y _\La $ from the left, we get the sequence of right $ \La $-modules
	\[
		0 \to { Y \otimes_\La L } \to { Y \otimes_\La ( \La \otimes_\Ga \La ) } \to { Y \otimes_\La \La } \to 0
	\]	
	which is exact since $ Y \otimes_\La - $ is right exact as a tensor functor and $\Tor_1^\La ( Y , \La  ) = 0 $. Note that $\e_Y$ is isomorphic to $1_Y \otimes \e_\La $ in the sense of the following commutative diagram
	\[\begin{tikzcd}
		{ Y \otimes_\La ( \La \otimes_\Ga \La  ) } & { Y \otimes_\La \La } \\
		{ Y \otimes_\Ga \La  } & { Y }
		\arrow["{1_Y \otimes \e_\La }", from=1-1, to=1-2]
		\arrow["\e_Y", shorten <=5pt, from=2-1, to=2-2]
		\arrow["\simeq"', from=1-1, to=2-1]
		\arrow["\simeq", from=1-2, to=2-2]
	\end{tikzcd}\]
	where both vertical maps are induced by the natural isomorphism $- \otimes_\La \La \cong \iden_{ \sMod \La } $. We conclude that $G \cong { - \otimes_\La L_\La } $ as functors.
	\end{proof}
	
	For the rest of the paper, we use the identification $ G = { - \otimes_\La L_\La } $.
	
	\section{Proof of Main Theorem}
	\label{sectionproofofmainthm}

	The aim of this section is to apply \cref{thm:findimcleftextab}{} to a bound quiver algebra $ \La = k Q / I $ where the admissible ideal $ I $ possesses a strict monomial Gr{\"o}bner basis for a fixed arrow $ \a $. This is possible as the algebras $ \La $ and $ \La / \la \ol{ \a } \ra $ form a cleft extension in this case, see \cref{propcleftextension}{}. Therefore, we only have to verify conditions (i) to (iv) of the aforementioned theorem.

	In what follows, the set $ \BQ $ of paths in $Q$ is equipped with a fixed admissible order. We use $ \NN = \ntip (I)$ to denote the set of paths that do not occur as the tip of an element in $ I $ with respect to this order. Recall that $ \NN + I $ is a $ k $-basis of $ \La $, and the set $ \NN $ enjoys the special property of being closed under subpaths.
	
	For any vertex $ i \in Q_0 $ we write $ e_i $ to denote the trivial path corresponding to this vertex. For any pair of vertices $ i , j \in Q_0 $, we write $ e_i \NN e_j $ to denote the subset of $ \NN $ consisting of the paths with source $ i $ and target $ j $. Note that $ e_i \NN e_j + I $ is a $ k $-basis of $ e_i \La e_j $. Similarly, we write $ e_i \NN $ (resp.\ $ \NN e_i $) to denote the paths in $ \NN $ that start (resp.\ end) at vertex $ i $, and the set $ e_i \NN + I $ (resp.\ $ \NN e_i + I $) is a $ k $-basis of $ e_i \La $ (resp.\ $ \La e_i $). More generally, for paths $ r_1 , r_2 \in \BQ $ we write $ r_1 \NN r_2 $ to denote the set of paths of the form $ r_1 u r_2 $ where $ u \in \NN $, and $ t( r_1 ) = s ( u ) $ and $ s ( r_2 ) = t ( u ) $; we use $ r_1 \NN $ and $ \NN r_2$ in a similar fashion. Observe that the set $ r_1 \NN r_2 + I $ generates $ r_1 \La r_2 $ as a $ k $-vector space, but is not $ k $-linearly independent in general.
	
	Lastly, if $ \b \in Q_1 $ is any arrow, then $ s(\b) $ and $ t(\b) $ denote the source and target of $ \b $, respectively, and $ \sss ( \b ) $ and $ \ttt ( \b ) $ denote the respective trivial paths in both $ k Q $ and $ \La = k Q / I $ without risk of confusion.

	For the rest of the section, we write $ \La = k Q / I $ to denote a bound quiver algebra, where the admissible ideal $ I $ possesses a strict $ \a $-monomial Gr{\"o}bner basis $ \GG = \GG_{p, q}$ for a fixed arrow $ \a $. Furthermore, we write $ \Ga $ for the quotient algebra $ \La / \la \ol{ \a } \ra $, and $\Ga $ will often be identified with the bound quiver algebra $ k Q^* / I^* $ introduced in \cref{propcleftextension}{}.
	
	We begin with a useful preliminary result.

	\begin{lem}
		\label{lem:preliminary}
		Multiplication by $ p $ from the right induces isomorphisms $ \La { \sss ( p ) } \isomto \La \ol{ p } $ and $ \Ga { \sss ( p ) } \isomto \Ga \ol{ p } $. Similarly, multiplication by $ q $ from the left induces isomorphisms $ { \ttt ( q ) } \La \isomto \ol{ q } \La $ and $ { \ttt ( q ) } \Ga \isomto \ol{ q } \Ga $. 
	\end{lem}

	\begin{proof}
		We show first that the map $\La { \sss ( p ) } \xrightarrow{ \, - \cdot p \, } \La \ol{ p }$ is injective, as it is obviously a surjective $ \La $-homomorphism. Recall that the set $ \NN { \sss ( p ) } + I $ is a $ k $-basis of $ \La { \sss ( p ) } $ and therefore, for every $ l \in \La { \sss ( p ) } $, there are coefficients $ \mu_s \in k $ for $ s $ ranging over $ \NN { \sss ( p ) } $  such that $ l = \sum_s \mu_s \cdot s + I $. Let $ l $ be in the kernel of the above homomorphism, that is
			\[
				z = \sum_s \mu_s \cdot s p \in I.
			\]
		If we assume that $ l $ is non-zero, then $ z $ is non-zero as not all coefficients $ \mu_s $ are zero and the paths $ s p $ are different from each other for different values of $ s $. In particular, we have that $ \tip ( z ) = s p $ for some fixed path $ s \in \NN { \sss ( p ) } $. Consequently, there is an element $ g \in \GG $ such that $ \tip ( g ) $ divides $ s p $. But $ \tip( g ) $ cannot divide $ s $ as the latter is in $ \NN $, and $ \tip ( g ) $ cannot divide $ p $ as property \hyperlink{P3}{(P3)} holds for $ \tip ( \GG ) $ by assumption. Therefore, we have $\tip ( g ) = s_2 p_1 $ where $ s_2 $ and $p_1$ are non-trivial subpaths of $ s $ and $ p $, respectively, such that $ s = s_1 s_2 $ and $ p = p_1 p_2 $. This implies in particular that $ p $ overlaps with $ \tip ( g ) $ from the left, a contradiction to property \hyperlink{P2}{(P2)} for $ \tip ( \GG ) $.
		
		Having established the result for $\La$, we now turn our attention to the corresponding isomorphisms over $\Ga$. For this, we establish first a $ k $-basis for $ \Ga $ related to the distinguished basis $ \NN + I $ for $ \La $.
		
		We begin by noting that $\GG$ is a generating set for $I$ satisfying property \hyperlink{P1}{(P1)}. Therefore, the quotient algebra $\Ga $ may be identified with the bound quiver algebra $k Q^* / I^*$ where $Q^* = Q \setminus \{ \a \}$ and $ I^* = k Q^* \cap I $ according to \cref{propcleftextension}{}. It follows from the proof of the same proposition that $\GG \setminus \{ p \a q \}$ is a generating set for $I^*$. Moreover, it can be proven that the total order on the set of paths in $Q^*$ inherited by the fixed order on $ \BQ $ is admissible, and the set $ \GG^* = \GG \setminus \{ p \a q \} $ is a Gr{\"o}bner basis for $I^*$ with respect to this order. The proof of the last two claims is omitted as they will not be used.
		
		We show next that $\NN_\Ga + I^*$ is a $k$-basis of $\Ga$, where $\NN_\Ga = \NN \cap k Q^* $. A key fact toward this goal is the decomposition
		\[
		I = I_\a \oplus I^*
		\]
		where $I_\a $ denotes the subspace of $ I $ consisting of $ k $-linear combinations of paths divided by the arrow $ \a $; see the proof of \cref{propcleftextension}{} for more details. To prove our claim, let $r $ denote a path in $Q^*$ with source $i$ and target $j$. Then regarding $ r $ as a path in $ Q $ yields
		\[
		r -  \sum_{ s \in e_i \NN e_j } \mu_s \cdot s \in  I
		\]
		for unique coefficients $\mu_s \in k $. Applying the above decomposition of $ I $, we get that
			\[
				z = r - \sum_{s \in e_i \NN_\Ga e_j } \mu_s \cdot s  \in I^*
			\]
		where now $ s $ ranges over $ e_i \NN_\Ga e_j $, as $r$ avoids $ \a $. In particular, we have that $r + I^* $ is generated by $\NN_\Ga + I^* $.
		To see that $ \NN_\Ga + I^* $ is also $ k $-linearly independent, assume that $z = \sum_{ s \in \NN_\Ga } \mu_s \cdot s \in I^* $ for coefficients $\mu_s \in k $. Viewing $ z $ as an element of $ k Q $, we have that $z \in I$ implying that $\mu_s = 0 $ for all $s \in \NN_\Ga $ since the set $ \NN_\Ga + I \subseteq \NN + I $ is $ k $-linearly independent.
		
		Turning now our attention to the map $ \Ga { \sss ( p ) } \xrightarrow{ \, - \cdot p \, } \Ga \ol{ p } $, we want to prove its injectivity as it is clearly a surjective $ \Ga $-homomorphism. First, we note that the set $ \NN_\Ga { \sss ( p ) } + I^*$ is a $ k $-basis of $  \Ga { \sss ( p ) } $ as an immediate consequence of the fact that $ \NN_\Ga + I^* $ is a $ k $-basis of $ \Ga $. It remains to observe that the inclusion $  \NN_\Ga p \subseteq \NN_\Ga { \ttt ( p ) } $ holds as a direct consequence of the inclusion $ \NN p \subseteq \NN { \ttt ( p ) } $, which has already been proven indirectly.
		Note that we use $ \ol{ p } $ to denote both $ p + I $ and $ p + I^* $.

		The proof for the statements about $ q $ is analogous and therefore omitted.
	\end{proof}

	Three out of the four conditions in \cref{thm:findimcleftextab}{} follow essentially from the next result, where we establish a projective resolution of the ideal $ \la \ol{ \a } \ra \subseteq \La $ as a $ \La $-$ \Ga $- and $ \Ga $-$ \La $-bimodule. Both $\Ga$-actions come from the restriction of scalars induced by the section algebra monomorphism $ \Ga \monicc \La $ established in \cref{propcleftextension}{}.

	\begin{lem}
			\label{lem:2}
		There are minimal projective resolutions
			\[
				0 \to  { \La { \sss ( p ) } } \otimes_k { { \ttt ( q ) } \Ga }  \to { \La { \sss ( \a ) } } \otimes_k { { \ttt ( \a ) } \Ga } \to \la \ol{ \a } \ra \to 0 
			\]
		and
			\[
				0 \to  { \Ga { \sss ( p ) } } \otimes_k { { \ttt ( q ) } \La }  \to { \Ga { \sss ( \a ) } } \otimes_k { { \ttt ( \a ) } \La } \to \la \ol{ \a } \ra \to 0 
			\]
		 of the ideal $ \la \ol{ \a } \ra $ as a $ \La $-$ \Ga $- and $ \Ga $-$ \La $-bimodule respectively, where the differentials are given by $ { \sss ( \a )  } \otimes_k {  \ttt ( \a ) } \mapsto \ol{ \a } $ and $ { \sss ( p )  } \otimes_k {  \ttt ( q ) } \mapsto \ol{ p } \otimes \ol{ q } $ in both resolutions. 
	\end{lem}

	\begin{proof}
		We begin by noting that $ \la \ol{ \a } \ra $ is generated as a $\La$-$\Ga$-bimodule by $ \ol{ \a } $. Indeed, it is generated as a $ k $-vector space by elements of the form $ \ol{ u \a v } $ where $ u $ and $ v $ are paths, and we may assume that $ v $ avoids $ \a $. Therefore the map $ f \colon { \La { \sss ( \a ) } } \otimes_k { { \ttt ( \a ) } \Ga } \to {}_\La \la \ol{ \a } \ra _\Ga $ defined by $ { \sss ( \a )  } \otimes_k {  \ttt ( \a ) } \mapsto \ol{ \a } $ is a surjective homomorphism of $ \La $-$ \Ga $-bimodules, as it sends $ u \otimes v $ to $ \ol{ u \a v } $. This homomorphism is in fact a projective cover as the $\La$-$\Ga$-bimodule $ { \La { \sss ( \a ) } } \otimes_k { { \ttt ( \a ) } \Ga } $ is projective and indecomposable. 
		
		It is obvious that $ \La \ol{ p } \otimes \ol{ q } \Ga $, i.e.\ the submodule of $ { \La { \sss ( \a ) } } \otimes_k { { \ttt ( \a ) } \Ga } $ generated by $ \ol{ p } \otimes \ol{ q } $, is in the kernel of $ f $ as $ p \a q \in I $. Note that we use $\ol{ q } $ to denote both $ q + I $ and its image in $ \Ga = \La / \la \ol{ \a } \ra $ under the canonical epimorphism $\La \epic \Ga $.
		
		For the opposite inclusion, let $ Z \in { \La { \sss ( \a ) } } \otimes_k { { \ttt ( \a ) } \Ga } $ be written as
			\[
				Z = \sum_{u , v}  \mu_{u , v} \cdot \ol{ u } \otimes \ol{ v }
			\]
		for unique coefficients $ \mu_{u , v} \in k $, where $ u $ and $ v $ range over $ \NN { \sss ( \a ) } $ and $ { \ttt ( \a ) } \NN_\Ga  $, respectively. Then $ f ( Z ) = \ol{ w } $ for
			\[	
				w = \sum_{u , v}  \mu_{u , v} \cdot u \a v
			\]
		where $ u $ and $ v $ range over the same subsets of $ \NN $ as above. Assume now that $ Z \in \Ker f $, that is $ w \in I $. For paths $ u , u' \in \NN { \sss ( \a ) } $ and $ v , v' \in { \ttt ( \a ) } \NN_\Ga  $, the equality $ u \a v = u' \a v' $ implies that $ u = u' $ and $ v = v' $ as the paths $ v $ and $ v' $ avoid $ \a $.
		Let $ w ' $ denote the element of $ k Q $ that results from $ w $ if we restrict the above sum to pairs $ u $ and $ v $ such that $ u $ is not divided by $ p $ form the right or $ v $ is not divided by $ q $ from the left. It holds that $ w ' \in I $ as $ p \a q \in I $. If $ w' $ is not zero, there are paths $ u_0 \in \NN { \sss ( \a ) } $ and $ v_0 \in { \ttt ( \a ) } \NN_\Ga $ such that $ u_0 $ is not divided by $ p $ from the right or $ v_0 $ is not divided by $ q $ from the left, and $ u_0 \a v_0 \in \tip ( I ) $. Let $ g $ be the element of $ \GG $ whose tip divides $ u_0 \a v_0 $. If $ g = p \a q $ then $ \tip ( g ) = p \a q $ has to divide $ u_0 $ or $ v_0 $, a contradiction as $ u_0 , v_0 \in \NN $. If $ g \neq p \a q$ then $ \tip ( g ) $ avoids $ \a $, which leads us back to the same contradiction. We deduce that $ w ' = 0 $, implying that $ Z \in \La \ol{ p } \otimes \ol{ q } \Ga $. Hence, we have established that $ \Ker f = \La \ol{ p } \otimes \ol{ q } \Ga $.
		
		It is not difficult to see now that the map $ g \colon \La { \sss ( p ) } \otimes_k { \ttt ( q ) } \Ga \to { \La \ol{ p } \otimes  \ol{ q } \Ga } $ defined by $ { \sss ( p ) } \otimes { \ttt ( q ) } \mapsto \ol{ p } \otimes \ol{ q } $ is a surjective homomorphism of $\La $-$\Ga $-bimodules. To see that $ g $ is bijective, observe that the set $ \NN p \otimes q \NN_\Ga $ is a $ k $-basis of $ { \La \ol{ p } \otimes  \ol{ q } \Ga } $ and, thus, the two bimodules have the same $ k $-dimension due to \cref{lem:preliminary}{}. This proves the validity of the first projective resolution of $ \la \ol{ \a } \ra $ as a $ \La $-$ \Ga $-bimodule.
		
		The proof for the second projective resolution is analogous.
		\end{proof}

		\begin{cor}
				\label{cor:1}
			For any module $ X_\Ga $ there is a projective resolution of the form
				\[
					 0 \to { \ttt ( q ) \La }^{ \dim_k X \ol{ p } } \to { \ttt ( \a ) \La }^{ \dim_k X { \sss ( \a ) } } \to H ( X ) \to 0
				\]
			for $ H ( X ) $ as a right $ \La $-module.
		\end{cor}

		\begin{proof}
			Recall that $ H = - \otimes_\Ga \la \ol{ \a } \ra _\La $ according to \cref{sectioncleftextabcat}{}. Tensoring the minimal projective resolution of the ideal $ \la \ol{ \a } \ra $ as a $\Ga$-$\La$-bimodule (see \cref{lem:2}{}) with $ X $ from the left gives an exact sequence of the form
				\[
					X \otimes_\Ga ( { \Ga { \sss ( p ) }  } \otimes_k { { \ttt ( q ) } \La } ) \to X \otimes_\Ga ( { \Ga { \sss ( \a ) } } \otimes_k { { \ttt ( \a ) } \La } ) \to H ( X ) \to 0 .
				\]
			Using the associativity of the tensor product, and the well-known identifications $ X \otimes_\Ga { \Ga { \sss ( p ) }  } \simeq X { \sss ( p ) } $ and $ X \otimes_\Ga  \Ga { \sss ( \a ) } \simeq X { \sss ( \a ) } $, we get an exact sequence
				\[
					X { \sss ( p ) } \otimes_k { { \ttt ( q ) } \La }  \to X { \sss ( \a ) } \otimes_k { { \ttt ( \a ) } \La }  \to H ( X ) \to 0
				\]
			where $ x { \sss ( p ) } \otimes { \ttt ( q ) } z \mapsto { x \ol{ p } } \otimes \ol{ q } z $ and $  x { \sss ( \a ) } \otimes { \ttt ( \a ) } z \mapsto x { \sss ( \a ) }  \otimes { \ol{ \a } z } $ for every $ x \in X $ and $ z \in \La $. The image of the left map is equal to $ X \ol{ p } \otimes \ol{ q } \La$, i.e.\ to the submodule of $ X { \sss ( \a ) } \otimes_k { { \ttt ( \a ) } \La } $ generated by $ \ol{ p } \otimes \ol{ q } $, and therefore isomorphic to $ { \ttt ( q ) \La }^{  \dim_k X \ol{ p } } $ as a right $ \La $-module. Similarly, the right $ \La $-module $ X { \sss ( \a ) } \otimes_k { { \ttt ( \a ) } \La } $ is isomorphic to the direct sum of $ \dim_k X { \sss ( \a ) } $ copies of $  \ttt ( \a ) \La  $, completing the proof.
		\end{proof}
		
		In our last lemma, we compute the projective dimension of the $\La$-bimodule $ L $ that occurs as the kernel of the map $ \La \otimes_\Ga \La \to \La $ induced by multiplication (see \cref{sectioncleftextabcat}{}) through an explicit minimal projective resolution.

		\begin{lem}
				\label{lem:3}
			There is a minimal projective resolution of the form
				\[
					0 \to  { \La { \sss ( p ) } } \otimes_k { { \ttt ( q ) } \La }  \to { \La { \sss ( \a ) } } \otimes_k { { \ttt ( \a ) } \La } \to L \to 0
				\]
			for $ L $ as a $ \La $-bimodule.
		\end{lem}
	
		\begin{proof}
			We claim that $ L $ is generated by the element $  { \ol{ \a } \otimes { \ttt ( \a ) } - { \sss ( \a ) }\otimes \ol{ \a } } $ as a $ \La $-bimodule. Begin by observing that $ L $ is generated by the set
				\begin{equation}
						\label{eq:4}	\hypertarget{eq:4}{}
					\{ l \otimes 1 - 1 \otimes l \mid l \in \La \}.
				\end{equation}
			Indeed, it is obvious that this set is included in $ L $. For the opposite inclusion, take an element $ Z = \sum_i l_i \otimes m_i \in L $, i.e.\ choose elements $l_i , m_i \in \La $ such that $ \sum_i l_i m_i = 0 $. Then we have
				\[
					Z = \sum_i l_i \otimes m_i - 1 \otimes 0 = \sum_i ( l_i  \otimes 1 ) m_i  - \sum_i ( 1 \otimes l_i ) m_i = \sum_i ( l_i \otimes 1 - 1 \otimes l_i ) m_i
				\]
			implying that $ Z \in {}_\La \la \{ l \otimes 1 - 1 \otimes l \mid l \in \La \} \ra_\La $.
			
			We show next that requiring $ l $ to be in $ \la \ol{ \a } \ra $ in \hyperlink{eq:4}{(\ref{eq:4}{})} yields a set that still generates $ L $. Indeed, we have that $\La = \Ga \oplus \la \ol{ \a } \ra $ according to \cref{propcleftextension}{}, since there is an algebra cleft extension $\Ga \monicc \La \epic \Ga $, where we have identified $ \Ga $ with its image inside $ \La $. In particular, for every $l \in \La $, we have $ l = \gamma + x $ for unique $ \g \in \Ga $ and $ x \in \la \ol{ \a } \ra $, implying that
				\[
					l \otimes 1 - 1 \otimes l = ( x \otimes 1 + \g \otimes 1 ) - ( 1 \otimes x + 1 \otimes \g ) =  x \otimes 1 - 1 \otimes x
				\]
			as the tensor product is over $ \Ga $.
			
			The last step is to show that for every path $ u $ divided by $ \a $ at least once, it holds that $ Z_u = \ol{ u } \otimes 1 - 1 \otimes \ol{ u } \in {}_\La \la \ol{ \a } \otimes { \ttt ( \a ) } - { \sss ( \a ) } \otimes \ol{ \a } \ra _\La $. Assume that $ u = v_1 \a v_2 $ where $ v_1 $ and $ v_2 $ are paths that avoid $ \a $. Then we have that
				\[
					Z_u = \ol{ v_1 } \, ( \ol{ \a } \otimes { \ttt ( \a ) } - { \sss ( \a ) } \otimes \ol{ \a } ) \, \ol{ v_2 } \in {}_\La \la \ol{ \a } \otimes { \ttt ( \a ) } - { \sss ( \a ) } \otimes \ol{ \a } \ra _\La
				\]
			as the paths $v_1 $ and $ v_2 $ may pass through the tensor product over $ \Ga $. Now assume that $ u = v_1 \a v_2 \a v_3 $ where the paths $ v_1, v_2  $ and $ v_3 $ avoid $ \a $. Similarly, we have
				\[
					Z_u = \ol{ v_1 } \, ( \ol{ \a } \otimes { \ttt ( \a ) } - { \sss ( \a ) } \otimes \ol{ \a } ) \, \ol{ v_2 \a v_3 }
					+ \ol{ v_1 \a v_2 } \, ( \ol{ \a } \otimes { \ttt ( \a ) } - { \sss ( \a ) } \otimes \ol{ \a } ) \, \ol{ v_3 }	
				\]
			implying that $ Z_u \in {}_\La \la \ol{ \a } \otimes { \ttt ( \a ) } - { \sss ( \a ) } \otimes \ol{ \a } \ra _\La $. An inductive argument shows that $Z_u \in {}_\La \la \ol{ \a } \otimes { \ttt ( \a ) } - { \sss ( \a ) } \otimes \ol{ \a } \ra _\La $ for any path $ u $ divided by $ \a $ at least once, which completes the proof of our claim.
			In particular, the homomorphism of $\La$-bimodules $ f \colon \La { \sss ( \a ) } \otimes_k { \ttt ( \a ) } \La \to L $ defined by $  { \sss ( \a ) } \otimes_k { \ttt ( \a ) } \mapsto \ol{ \a } \otimes { \ttt ( \a ) } - { \sss ( \a ) } \otimes \ol{ \a } $ is surjective, and thus a projective cover as the $ \La $-bimodule $ \La { \sss ( \a ) } \otimes_k { \ttt ( \a ) } \La $ is projective and indecomposable.
			
			It is not difficult to see that $ \La \ol{ p } \otimes \ol{ q } \La $, i.e.\ the submodule of $ \La { \sss ( \a ) } \otimes_k { \ttt ( \a ) } \La $ generated by $ \ol{ p } \otimes \ol{ q } $, is contained in $ \Ker f $ as the paths $ p $ and $ q $ avoid $ \a $ by assumption. To see that this is exactly the kernel of $ f $, observe that $ L \simeq \La \otimes_\Ga \la \ol{ \a } \ra = H ( \La _\Ga ) $ as $ k $-vector spaces. This follows from the definition of $ L $, the decomposition $ \La = \Ga \oplus \la \ol{ \a } \ra $ of $ \La $ as a left and right $ \Ga $-module, and the well-known identifications $ \Ga \otimes_\Ga \la \ol{ \a } \ra \simeq \la \ol{ \a } \ra \simeq \la \ol{ \a } \ra \otimes_\Ga \Ga $ and $ \Ga \otimes_\Ga \Ga \simeq \Ga $. Furthermore, we have
				\[
					\dim_k H( \La_\Ga ) = { \dim_k \La { \sss ( \a ) } } \cdot { \dim_k { \ttt ( \a ) } \La } - { \dim_k \La { \ol{ p } } } \cdot { \dim_k { \ttt ( q ) } \La }
				\]
			from \cref{cor:1}{}, where 
			$ \dim_k { \ttt ( q ) } \La = \dim_k \ol{ q } \La $ according to \cref{lem:preliminary}{}. We deduce that $ \Ker f = \La \ol{ p } \otimes \ol{ q } \La $ as $ \dim_k \La { { \ttt ( \a ) } \otimes_k { \sss ( \a ) } \La } = { \dim_k \La { \ttt ( \a ) } } \cdot { \dim_k { \sss ( \a ) } \La } $ and $ \dim_k \La \ol{ p } \otimes \ol{ q } \La = { \dim_k \La \ol{ p } } \cdot { \dim_k \ol{ q } \La } $. Lastly, the homomorphism of $ \La $-bimodules $ \La { \sss ( \a ) } \otimes_k { \ttt ( \a ) } \La \to \La \ol{ p } \otimes \ol{ q } \La $ defined by $ { \sss ( \a ) } \otimes { \ttt ( \a ) } \mapsto \ol{ p } \otimes \ol{ q } $ is clearly surjective, and in fact bijective according to \cref{lem:preliminary}{}. This completes the proof.
		\end{proof}

	We are now ready to prove the main result of the paper.

	\begin{thm}
		\label{mainthm}
		Let $ I $ be an admissible ideal of the path algebra $ kQ $ possessing a strict $ \a $-monomial Gr{\"o}bner basis with respect to some admissible order on the set of paths in $Q$. Then
			\[
				\fpd \La \leq \fpd { \La / \la \ol{ \a } \ra } + 2.
			\]
	\end{thm}

	\begin{proof}
		Assume that $ \GG = \GG_{ p , q } $ is a strict $ \a $-monomial Gr{\"o}bner basis for $ I $. In particular, the set $ \GG $ is an $ \a $-monomial generating set for $ I $, and thus there is a cleft extension of algebras $ \Ga = \La / \la \ol{ \a } \ra  \monicc \La \epic \Ga $ according to \cref{propcleftextension}{}. Therefore it suffices to prove conditions (i) to (iv) of \cref{thm:findimcleftextab}{} for this cleft extension in order to apply the theorem.
		
		For condition (i), observe that
			\[
				p_{ \smod \La } = \sup\{ \pd \me(P)_\Ga \mid P \in \Proj( \smod \La ) \} = \pd \La_\Ga
			\]
		as every projective $\La$-module is isomorphic to a direct sum of direct summands of $ \La_\La $, and recall that we have a direct sum decomposition $\La_\Ga = \Ga \oplus \la \ol{ \a } \ra $. Moreover, it holds that $ \pd \la \ol{ \a } \ra_\Ga \leq 1 $ as the first projective resolution of \cref{lem:2}{} may be viewed as a projective resolution of right $ \Ga $-modules, isomorphic to
			\[
					0 \to  { { \ttt ( q ) } \Ga }^{ \dim_k { \La { \sss ( p ) } } }  \to { { \ttt ( \a ) } \Ga }^{ \dim_k { \La { \sss ( \a ) } } } \to \la \ol{ \a } \ra _\Ga \to 0 .
			\]
		We deduce that $ p_{ \smod \La } \leq 1 $.
		
		Similarly, for the second condition we have that
			\[
				 p_{ \smod \Ga } = \sup\{ \pd \mi(P')_\La \mid P' \in \Proj( \smod \Ga ) \} = \pd \Ga_\La .
			\]
		Observe that the natural epimorphism $ \pi \colon \La \epic \Ga $ is a projective cover of right $ \La $-modules as its kernel, the ideal $ \la \ol{ \a } \ra $, is contained in the Jacobson radical of $ \La $ which is equal to the radical of the regular $ \La $-module. In particular, we have that $ \pd \Ga_\La = \pd \la \ol{ \a } \ra_\La + 1$. It remains to observe that $ \pd \la \ol{ \a } \ra_\La \leq 1 $ due to the second projective resolution of \cref{lem:2}{} viewed as a projective resolution of right $ \La $-modules. We deduce that $ p_{ \smod \Ga } \leq 2 $.
		
		For condition (iii), observe that
			\[
				n_H = \sup\{ \pd H( X )_\La \mid X \in \smod \Ga \} \leq 1
			\]
		according to \cref{cor:1}{}. Using the minimal projective resolution of $ K $ as a $ \La $-bimodule established in \cref{lem:3}{}, on may construct a projective resolution of $ G( Y_\La )_\La $ of length one for every module $ Y_\La $ in the same way as we did for $ H ( X )_\La $ in the proof of \cref{cor:1}{}. We deduce that
			\[
				n_G = \sup \{ \pd G( Y )_\La \mid Y \in \smod \La \} \leq 1
			\]
		and the proof is complete.
	\end{proof}

	\begin{cor}
			\label{maincor}
		Let the quotient algebra $ \La / \la \ol{ \a } \ra $ be a strict monomial arrow removal of the bound quiver algebra $ \La = k Q / I $. Then 
			\[
				\fpd \La \leq \fpd { \La / \la \ol{ \a } \ra } + 2.
			\]
	\end{cor}

	\begin{proof}
		Immediate consequence of \cref{mainthm}{} and \cref{lem:1}{}.
	\end{proof}

	\section{Examples}
	\label{sectionexamples}
	
	This section is devoted to giving various concrete examples illustrating the
	reduction techniques we have discussed in the previous sections. In
	this context, we recall the next notion from \cite{GPS}. Furthermore, we provide an example showcasing why condition \hyperlink{c}{(c)} in \hyperlink{defn:1}{\cref{defn:1}{}.(iii)}{} is necessary in the context of strict monomial arrow removal.
	
	\begin{defn}
		\label{defn:reduced}
		A bound quiver algebra $ \La = k Q / I $ is called \emph{reduced} if: 
		\begin{enumerate}[\rm(i)]
			\item all arrows occur in every generating set for $ I $;  
			\item all simple modules have infinite injective dimension and
			projective dimension at least $2$; 
			\item no triangular reductions are possible.
		\end{enumerate}
	\end{defn}
	
	We proceed with our first algebra, the running example of \cref{sec:mon.ar.rem}{}.

	\begin{exam}
		\label{exam:2}
		Let $ \La = k Q / I $ be the bound quiver algebra of \cref{exam:1}{}. We recall that $ Q $ is the quiver below and $ k $ is any field. Furthermore, the ideal $ I $ is generated by set $ T $ consisting of:
		\begin{enumerate}[\rm(i)]
			\item the paths $\theta \a \b \g_1 \d_1 $ and $ \e \z $, and $\g_2 \d_2 - \g_1 \d_1 $;
			\item the paths of the form $\l_i^{j_i}$, where we have fixed integers $j_i \geq 2 $ for $i = 1, 2, \ldots, 9 $;
			\item all the paths of length two involving one loop and one non-loop arrow except for the paths $\l_1 \a $, $ \a \l_2 $, $\d_1 \l_6 $, $\d_2 \l_6 $ and $\l_9 \theta $.
		\end{enumerate}

		
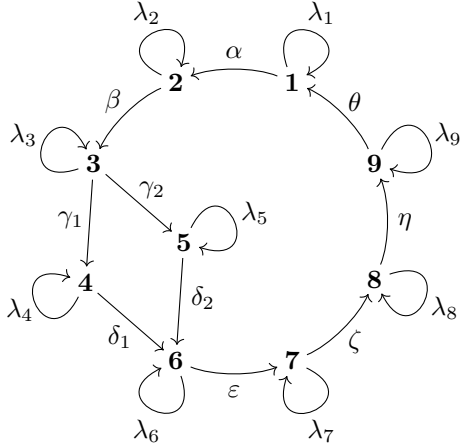
\smallskip
		\begin{figure}[h!]
			\centering
			\begin{tikzpicture}	


				\node at ($(0,0)+(67.5:2)$) {$\mathbf{1}$};
				
				\node at ($(0,0)+(112.5:2)$) {$\mathbf{2}$};
				
				\node at ($(0,0)+(157.5:2)$) {$\mathbf{3}$};

				\path 
				($(0,0)+(157.5:2)$)			coordinate (3)      
				($(0,0)+(247.5:2)$)			coordinate (6)
				;
				
				\coordinate (Mid 3 6) at ($(3)!0.5!(6)$);

				\node at ($(Mid 3 6)+(202.5:0.7)$) {$\mathbf{4}$};
				
				\node at ($(Mid 3 6)+(22.5:0.7)$) {$\mathbf{5}$};

				\node at ($(0,0)+(247.5:2)$) {$\mathbf{6}$};
				
				\node at ($(0,0)+(292.5:2)$) {$\mathbf{7}$};
				
				\node at ($(0,0)+(337.5:2)$) {$\mathbf{8}$};
				
				\node at ($(0,0)+(22.5:2)$) {$\mathbf{9}$};


				\draw[->,shorten <=6pt, shorten >=6pt] ($(0,0)+(67.5:2)$) arc (67.5:112.5:2);
				
				\draw[->,shorten <=6pt, shorten >=6pt] ($(0,0)+(112.5:2)$) arc (112.5:157.5:2);

				\draw[->,shorten <=6pt, shorten >=6pt]  ($(0,0)+(157.5:2)$) -- ($(Mid 3 6)+(202.5:0.7)$);
				
				\draw[->,shorten <=6pt, shorten >=6pt]  ($(0,0)+(157.5:2)$) -- ($(Mid 3 6)+(22.5:0.7)$);

				\draw[->,shorten <=6pt, shorten >=6pt] ($(Mid 3 6)+(202.5:0.7)$) -- ($(0,0)+(247.5:2)$);
				
				\draw[->,shorten <=6pt, shorten >=6pt] ($(Mid 3 6)+(22.5:0.7)$)  -- ($(0,0)+(247.5:2)$);

				\draw[->,shorten <=6pt, shorten >=6pt] ($(0,0)+(247.5:2)$) arc (247.5:292.5:2);
				
				\draw[->,shorten <=6pt, shorten >=6pt] ($(0,0)+(292.5:2)$) arc (292.5:337.5:2);
				
				\draw[->,shorten <=6pt, shorten >=6pt] ($(0,0)+(337.5:2)$) arc (337.5:382.5:2);
				
				\draw[->,shorten <=6pt, shorten >=6pt] ($(0,0)+(22.5:2)$) arc  (22.5:67.5:2);


				\draw[->,shorten <=6pt, shorten >=6pt] ($(0,0)+(67.5:2)$).. controls +(67.5+41:1.2) and +(67.5-41:1.2) .. ($(0,0)+(67.5:2)$) ;
				\node at ($(0,0)+(67.5:3)$) {$\l_1$};
				
				\draw[->,shorten <=6pt, shorten >=6pt] ($(0,0)+(112.5:2)$).. controls +(112.5+41:1.2) and +(112.5-41:1.2) .. ($(0,0)+(112.5:2)$) ;
				\node at ($(0,0)+(112.5:3)$) {$\l_2$};
				
				\draw[->,shorten <=6pt, shorten >=6pt] ($(0,0)+(157.5:2)$).. controls +(157.5+41:1.2) and +(157.5-41:1.2) .. ($(0,0)+(157.5:2)$) ;
				\node at ($(0,0)+(157.5:3)$) {$\l_3$};

				\draw[->,shorten <=6pt, shorten >=6pt] ($(Mid 3 6)+(202.5:0.7)$).. controls +(202.5+41:1.2) and +(202.5-41:1.2) .. ($(Mid 3 6)+(202.5:0.7)$) ;
				\node at ($(0,0)+(202.5:3.05)$) {$\l_4$};
				
				\draw[->,shorten <=6pt, shorten >=6pt] ($(Mid 3 6)+(22.5:0.7)$).. controls +(22.5+41:1.2) and +(22.5-41:1.2) .. ($(Mid 3 6)+(22.5:0.7)$) ;
				\node at ($(0,0)+(22.5:0.3)$) {$\l_5$};

				\draw[->,shorten <=6pt, shorten >=6pt] ($(0,0)+(247.5:2)$).. controls +(247.5+41:1.2) and +(247.5-41:1.2) .. ($(0,0)+(247.5:2)$) ;
				\node at ($(0,0)+(247.5:3)$) {$\l_6$};
				
				\draw[->,shorten <=6pt, shorten >=6pt] ($(0,0)+(292.5:2)$).. controls +(292.5+41:1.2) and +(292.5-41:1.2) .. ($(0,0)+(292.5:2)$) ;
				\node at ($(0,0)+(292.5:3)$) {$\l_7$};
				
				\draw[->,shorten <=6pt, shorten >=6pt] ($(0,0)+(337.5:2)$).. controls +(337.5+41:1.2) and +(337.5-41:1.2) .. ($(0,0)+(337.5:2)$) ;
				\node at ($(0,0)+(337.5:3)$) {$\l_8$};
				
				\draw[->,shorten <=6pt, shorten >=6pt] ($(0,0)+(22.5:2)$).. controls +(22.5+41:1.3) and +(22.5-41:1.3) .. ($(0,0)+(22.5:2)$) ;
				\node at ($(0,0)+(22.5:3.05)$) {$\l_9$};


				\node at ($(0,0)+(90:2.25)$) {$\alpha$};
				
				\node at ($(0,0)+(135:2.25)$) {$\beta$};

				\coordinate (gamma 1) at ($(3)!0.5!($(Mid 3 6)+(202.5:0.7)$)$);
				
				\coordinate (gamma 2) at ($(3)!0.5!($(Mid 3 6)+(22.5:0.7)$)$);
				
				\coordinate (delta 1) at ($($(Mid 3 6)+(202.5:0.7)$)!0.5!(6)$);
				
				\coordinate (delta 2) at ($($(Mid 3 6)+(22.5:0.7)$)!0.5!(6)$);

				\node at ($(gamma 1)+(170:0.25)$) {$\gamma_1$};
				
				\node at ($(gamma 2)+(45:0.25)$) {$\gamma_2$};

				\node at ($(delta 1)+(235:0.25)$) {$\delta_1$};
				
				\node at ($(delta 2)+(10:0.31)$) {$\delta_2$};

				\node at ($(0,0)+(270:2.25)$) {$\varepsilon$};			
				
				\node at ($(0,0)+(315:2.25)$) {$\zeta$};
				
				\node at ($(0,0)+(0:2.25)$) {$\eta$};
				
				\node at ($(0,0)+(45:2.25)$) {$\theta$};

			\end{tikzpicture}	
		\end{figure}
		
		Recall that the quotient algebra $ \La / \la \ol{ \a } \ra $ is not a strict monomial arrow removal of $ \La $, see \cref{exam:1'}{}. However, there is a strict $ \a $-monomial basis $ \GG = \GG_{ p , q } $ for $ I $, where $p = \theta $ and $ q = \b \g_1 \d_1 $ or $ q = \b \g_2 \d_2 $, for any admissible order on the set of paths in $ Q $, see \cref{exam:1''}{}. We show next that $ \La $ is a reduced algebra.
		
		First, all simple modules have infinite projective and injective dimension due to the loops $\l_i$ according to the Strong No Loop Theorem \cite{ILP}{}. Furthermore, the algebra $\La$ is triangle reduced as the quiver $Q$ is strongly connected, see \cite[Lemma~3.12]{Giata}{}. Lastly, every arrow of $Q$ occurs in every generating set $ T $ for $I$. Indeed, every arrow different from $ \a $ occurs in $ I $ in a path of length $ 2 $ and, therefore, it must occur in $T$ as well. On the other hand, if $ \a $ did not occur in $T$, then $\la \ol{ \a } \ra $ would be a projective $ \La $-bimodule according to \cite{GPS}{}, and this would contradict the fact that $ I $ possesses a strict $\a$-monomial Gr{\"o}bner basis. To see that, note that the homomorphism of $ \La $-bimodules $ \La { e_1 } \otimes_k { e_2 } \La \to \la \ol{ \a } \ra$ such that $ { e_1 } \otimes { e_2 } \mapsto \ol{ \a } $ is a projective cover, and its kernel contains the non-zero element $ \ol{ p } \otimes \ol{ q } $.

		We also note that $\fpd \La \neq 0 $ according to \cite[Lemma~6.2]{Bass}{}, since the left regular module $ {}_\La \La $ does not possess a submodule isomorphic to the simple left module corresponding to vertex $ 2 $. Indeed, no path occurring in $ T $ starts with $ \a $, and thus $z \in I $ whenever $z \in { e_2 { k Q } } $ and $  \a z \in I $. Furthermore, the algebra $ \La $ is not monomial, and its Loewy length can be arbitrarily large as integers $ j_i \geq 2 $ may be chosen freely. Finally, if we choose $j_7=2$ (i.e.\ $\l_7^2 \in I$), then \cite[Corollary~4.14]{Giata}{} applied to the loop $ \l_7 $ and the arrow $ \e $ implies that $ \La $ is not Iwanaga-Gorenstein, since the paths $ \e \l_7 $ and $ \e \z $ are zero in $ \La $.

		We close the example by demonstrating that
		\[
		0 < \fpd \La \leq 9
		\]
		through an application of \cref{mainthm}{}. Indeed, \cref{mainthm}{} implies directly that $ \fpd \La \leq \fpd { \La / \la \ol{ \a } \ra } + 2 $. Furthermore, we perform seven consecutive triangle reductions for the quotient algebra $ \La / { \la \ol{ \a } \ra } $ (see \cite{FosGriRei}). At every step, we extract a local algebra corresponding to the external vertex we want to remove (the unique source or sink vertex), or the direct product of two local algebras corresponding to vertices $4$ and $5$, until we reach two local algebras corresponding to the remaining two vertices. As each such reduction can increase the finitistic dimension by at most one, and local algebras have zero finitistic dimension, this process implies that $ \fpd { \La / \la \ol{ \a } \ra } \leq 7 $. Consequently, the upper bound $ \fpd \La \leq 9 $ holds.
	\end{exam}

	The next example is a variation of \cref{exam:2}{}.
	
	\begin{exam}
		\label{exam:3}
		Let $Q$ be the same quiver as in \cref{exam:2}{}, and let $ T' $ be the set that results from the set $ T $ of the same example if we replace $\th \a \b \g_1 \d_1 $ with $\th \a \b $. It holds that $ \La' = k Q / I' $ is a bound quiver algebra, where $ I' $ is the ideal of $ k Q $ generated by $ T' $. Furthermore, the set $ T' $ is strict $ \a $-monomial as $ p \a q \in T' $ for the paths $ p = \th $ and $q = \b $, and $ T' $ satisfies properties \hyperlink{P1}{(P)} for these paths. For instance, property \hyperlink{P3}{(P3)} is automatically satisfied as both $p$ and $q$ are arrows. 
		
		The algebra $ \La' $ is reduced, it is not monomial, its Loewy length is arbitrarily large and its little finitistic dimension is non-zero, in the same way as in \cref{exam:2}{}. Moreover, the algebra $ \La ' $ is not Iwanaga-Gorenstein if we choose $ j_7 = 2 $.
		It follows from \cref{maincor}{} and the above paragraph that $0 < \fpd  \La' \leq 9 $.
	\end{exam}

	\begin{exam}
			\label{magicexam}
		Let $Q$ be the quiver given by 
		
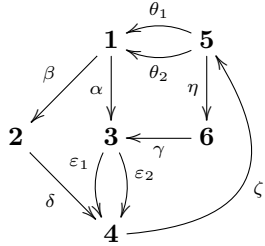
\[\xymatrix{
			& \mathbf{1} \ar[dl]_{\b} \ar[d]_{\a} &
			\mathbf{5} \ar@/^-0.5pc/[l]_{\th_1} \ar@/^0.5pc/[l]^{\th_2} \ar[d]_{\h} \\
			\mathbf{2} \ar[dr]_{\d} & \mathbf{3} \ar@/^-0.5pc/[d]_(0.3){\e_1} \ar@/^0.5pc/[d]^(0.4){\e_2} & \mathbf{6} \ar[l]^{\g} \\
			& \mathbf{4} \ar@/^-3.0pc/[uur]_{\z} & 
		}\]
		with relations
		\[ T = \{ \th_1 \a \e_1, \, \e_2 \z, \,  \z \h \g, \, \z \th_2, \, \b \d \z, \, \z \th_1 \b, \, \d \z \h, \, \g \e_1 - \g \e_2, \, \h \g \e_2 \} \]
		over $\mathbb{Z}_2$.  Let $I$ be the ideal of $\mathbb{Z}_2Q$
		generated by these relations, and let $\Lambda = \mathbb{Z}_2Q/I$. We
		show that the finitistic dimension of $ \Lambda$ is finite, even though it is a reduced algebra in the sense of \cref{defn:reduced}{}.

		First, we argue that all simple modules have projective
		dimension at least $2$.  Let $S_i$ denote the simple module associated
		to vertex $i$ for $i = 1,2,\ldots, 6$.  Computations then show that
		\begin{align}
			\pd_\Lambda S_2 & = 4,\notag\\
			\pd_\Lambda S_6 & = 2,\notag\\
			\Omega^4_\Lambda(S_1) & \simeq \Omega^2_\Lambda(S_1)\oplus S_4,\notag\\
			\Omega^4_\Lambda(S_3) & \simeq \Omega^2_\Lambda(S_1) \oplus Z_3	\notag	\\
			\Omega^3_\Lambda(S_4) & \simeq \Omega^2_\Lambda(S_1)\oplus Z_4	\notag	\\
			\Omega^5_\Lambda(S_5) & \simeq \Omega^2_\Lambda(S_1)\oplus Z_5		\notag
		\end{align}
		for modules $ Z_3 , Z_4 $ and $ Z_5 $. The first isomorphism implies that $ \pd_\La \Omega^2_\Lambda(S_1) = \infty $, and together with the other three isomorphisms it yields that $\pd_\Lambda S_i = \infty$ for $i =
		1,3,4,5$. In particular, it holds that $ \pd S_i \geq 2 $ for every vertex $ i $.
				
		Next, we show that all simple modules have infinite injective
		dimension.  Computations show that
		\[\Omega^{-2}_\Lambda(S_1) \simeq \Omega^{-2}_\Lambda(S_2) \simeq
		\Omega^{-5}_\Lambda(S_3) \simeq \Omega^{-1}_\Lambda(S_5)\simeq
		\Omega^{-4}_\Lambda(S_6),\] 
		\[\Omega^{-3}_\Lambda(S_5) \simeq \Omega^{-1}_\Lambda(S_5) \oplus X\]
		\[\Omega^{-2}_\Lambda(S_4) \simeq S_5\oplus Y\]
		for modules $X$ and $Y$. We deduce from the isomorphism of the second row that $ \idim \Omega^{-1}_\Lambda(S_5) = \idim S_5 = \infty $. Our claim follows now from the rest of the provided isomorphisms.
		
		In addition to the above, all arrows of $ Q $ occur in every generating set for $ I $, and no triangular reduction is possible as the quiver $ Q $ is strongly connected, see \cite[Lemma~3.12]{Giata}{}.
		Moreover, it holds that $ \La $ is not monomial as the paths $ \g \e_1 $ and $ \g \e_2 $ are equal and non-zero, the Loewy length of $ \La $ is $ 8 $ where $ \th_2 \a \e_1 \z \th_1 \a \e_2 $ is the longest non-zero path, and $ \fpd \La $ is non-zero as the simple left $ \La $-module corresponding to vertex $ 3 $, for instance, is not isomorphic to a submodule of $ {}_\La \La $ (\!\!\cite[Lemma~6.2]{Bass}{}).

		We show next that the algebra $ \La $ admits the quotient algebra $ \La / \la \ol{ \a } \ra $ as a strict
		monomial arrow removal. Specifically, the generating set $ T = T_{ p , q } $ is strict $ \a $-monomial for $ p = \th_1 $ and $ q = \e_1 $ according to \cref{lem:0}{}. Indeed, property \hyperlink{P1}{(P1)} clearly holds for $ T $ and $ \th_1 \a \e_1 \in T $. Property \hyperlink{P3}{(P3)} also holds as no path occurring in $ T $ may divide $ p $ or $ q $, being the case that they comprise a single arrow each. Furthermore, properties \hyperlink{P2}{(P2)} and \hyperlink{P2op}{(P2)$^{\textrm{op}}$} hold as no path occurring in $ T $ starts with arrow $ \e_1 $ or ends with arrow $ \th_1 $.

		Now let $Q^*$ be the quiver that results from quiver $Q$ after removing arrow $ \a $, and let
		$I^*$ be the ideal of $\mathbb{Z}_2Q^*$ generated by $ T \setminus \{ \th_1 \a \e_1 \} $.
		Then the quotient algebra $\Gamma = \Lambda/\langle \ol{ \a } \rangle $ may be identified with the bound quiver algebra $ \mathbb{Z}_2Q^*/I^* $. We complete the example by showing that $ \Ga $ is also a reduced algebra with finite finitistic dimension.

		First, one can compute that the projective and the injective dimensions of the 
		simples modules are unchanged. Furthermore, all arrows in the quiver of $\Gamma$ occur in every generating set for $ I^* $, in a similar way as for $ \La $. Lastly, no triangular reduction is possible as the quiver $ Q^* $ is strongly connected (see \cite[Lemma~3.12]{Giata}{}).
		However, having removed relation $ \th_1 \a \e_1 $, we may now replace the arrows $ \e_1 $ and $ \e_2 $ of $ \La $ by the new arrows $ \e_1' = \e_1 - \e_2 $ and $ \e_2' = \e_2 $ as they still form a $ k $-basis for $ { e_3 J( \Ga ) e_4 } / { e_3 J( \Ga )^2 e_4 } $ where $ J( \Ga ) $ is the Jacobson radical of $ \Ga $. In particular, $ \Ga $ is isomorphic to a monomial bound quiver algebra, and thus $ \fpd \Ga < \infty $ according to \cite{GKK}. We deduce that $ \fpd \La < \infty $ by an application of \cref{mainthm}{}.
		
		All computations in this example have been carried out using QPA \cite{QPA}.
	\end{exam}

	We close this paper with an example that showcases why property \hyperlink{c}{(c)} in \hyperlink{defn:1}{\cref{defn:1}{}.(iii)}{} is necessary.

	\begin{exam}
		\label{exam:4}{}
		Let $Q$ be the following quiver
		
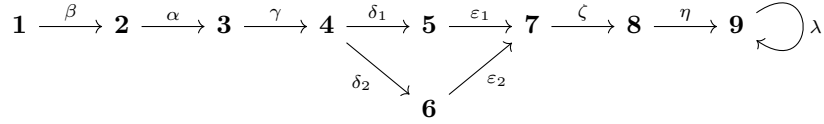
\[\begin{tikzcd}
			\mathbf{1} & \mathbf{2} & \mathbf{3} & \mathbf{4} & \mathbf{5} & \mathbf{7} & \mathbf{8} & \mathbf{9} \\
			&&&& \mathbf{6}
			\arrow["\b", from=1-1, to=1-2]
			\arrow["\a", from=1-2, to=1-3]
			\arrow["\g", from=1-3, to=1-4]
			\arrow["{\d_1}", from=1-4, to=1-5]
			\arrow["{\d_2}"', from=1-4, to=2-5]
			\arrow["{\e_1}", from=1-5, to=1-6]
			\arrow["\z", from=1-6, to=1-7]
			\arrow["\h", from=1-7, to=1-8]
			\arrow["\lambda", from=1-8, to=1-8, loop, in=325, out=35, distance=10mm]
			\arrow["{\e_2}"', from=2-5, to=1-6]
		\end{tikzcd}\]
		and let $ k $ be any field. Let $ I $ be the ideal of $ k Q $ generated by the set
			\[
				T = \{ \b \a \g \d_1 \e_1 \z , \,  \d_1 \e_1 - \d_2 \e_2 , \,  \e_2 \z \h , \, \lambda^2 \}
			\]
		and observe that $ \La = k Q / I $ is a bound quiver algebra. Furthermore, the set $ T $ satisfies all conditions of Definition \ref{defn:1}{} for $ p = \b $ and $ q = \g \d_1 \e_1 \z $ except for the last one, i.e.\ the path $ \d_1 \e_1 $ occurs in $ T $ and divides $ q $. For instance, it can be verified that no proper subpath of $ \b \a \g \d_1 \e_1 \z $ is in $ I $ and no path occurring in $ T $ ends with $ \b $ or starts with $ \z $, $ \e_1 \z $, $ \d_1 \e_1 \z $ or $ \g \d_1 \e_1 \z $.
		
		We claim that $ I $ does not possess a strict $ \a $-monomial Gr\"{o}bner basis.
		If there was such a basis, we would have that $ \pd \la \ol{ \a } \ra_\La \leq 1 $ according to the second projective resolution of \cref{lem:2}{} viewed as a complex of right $ \La $-modules. However, it holds that the beginning of the minimal projective resolution of $ \la \ol{ \a } \ra_\La $ is of the form
			\[
				0 \to S_9  \to e_8 \La \to e_3 \La ^{ 2 } \to \la \ol{ \a } \ra \to 0
			\]
		where $ S_9 $ is the simple right $ \La $-module corresponding to vertex $ 9 $. But $\pd S_9 = \infty $ due to the Strong No Loop Theorem \cite{ILP}{} and our claim follows. In particular, \cref{mainthm}{} does not apply for $ \La $ and the arrow $ \a $.
	\end{exam}


\begin{thebibliography}{19}
		
		\bibitem{ASS}
		\textsc{I.~Assem, D.~Simson, A.~Skowronski},
		\emph{Elements of the Representation Theory of Associative Algebras, Volume 1: Techniques of Representation Theory},
		University Press, Cambridge, 2006.
		
		
		\bibitem{Bass}
		\textsc{H.~Bass},
		\emph{Finitistic dimension and a homological generalization of semi-primary rings},
		Trans.\ Amer.\ Math.\ Soc.\ \textbf{95} (1960), no.\ 3, 466--488.
		
		
		\bibitem{Beligiannis}
		\textsc{A.~Beligiannis},
		\emph{Cleft extensions of abelian categories and applications to ring theory},
		Commun.\ Algebra \textbf{28} (2000), no.\ 10, 4503--4546.

		
		\bibitem{EPS}
		\textsc{K.~Erdmann, C.~Psaroudakis, \O.\ Solberg}, {\em Homological invariants of the arrow removal operation}, Represent.\ Theory \textbf{26} (2022), 370--387.
		
		
		\bibitem{FosGriRei} 
		\textsc{R.~Fossum, P.~Griffith, I.~Reiten}, \emph{Trivial
			Extensions of Abelian Categories with Applications to Ring Theory},
		Lecture Notes in Mathematics, vol. {\bf 456}, Springer, Berlin, 1975. 

		\bibitem{Giata}
		\textsc{O.~Giatagantzidis},
		\emph{Arrow reductions for the finitistic dimension conjecture}, arXiv:2507.12978 (2025), 51 pages.
				
		
		\bibitem{Green2}
		\textsc{E.L.~Green},
		\emph{Noncommutative Gr{\"o}bner bases, and
			projective resolutions}, in: P.~Dr{\"a}xler, C.M.~Ringel, G.O.~Michler (eds), Computational Methods for Representations of Groups and Algebras, Progr.\ Math.\ \textbf{173}, Birkh{\"a}user, Basel, 1999, 29--60.
		
		
		
		\bibitem{GKK} \textsc{E.\ L.\ Green,  E.\ Kirkman, J.\ Kuzmanovich},
		\emph{Finitistic dimension conjecture for monomial algebras}, J.\ Algebra
		\textbf{136}, no.\ 1, 37--50.
	 
		
		\bibitem{GPS} \textsc{E.\ L.\ Green, C.\ Psaroudaki, \O.\ Solberg},
		\emph{Reduction techniques for the finitistic dimension}, Trans.\ Amer.\ Math.\ Soc.\ \textbf{374} (2021), no.\ 10, 6839--6879.
		
		

		
		\bibitem{ILP}
		\textsc{K.~Igusa, S.~Liu, C.~Paquette},
		\emph{A proof of the strong no loop conjecture},
		Adv.\ Math.\ \textbf{228} (2011), no.\ 5, 2731--2742.
				
		\bibitem{QPA}
		\textsc{The QPA-team}, QPA - Quivers, path algebras and
		representations - a GAP package, Version 1.34; 2023
		(https://folk.ntnu.no/oyvinso/QPA/)
		
	\end{thebibliography}
\end{document}